\documentclass[twoside,11pt]{article}%
\usepackage{latexsym}
\usepackage{amsmath,amssymb,epsfig,subfigure}
\usepackage{amsthm,enumerate,verbatim}
\usepackage{amsfonts}
\usepackage{color}
\usepackage{amsmath}
\usepackage{graphicx}
\usepackage{epstopdf}
\usepackage{cite}
\providecommand{\U}[1]{\protect\rule{.1in}{.1in}}
\providecommand{\U}[1]{\protect\rule{.1in}{.1in}}
\newcommand{\BE}{\begin{equation}}
\newcommand{\EE}{\end{equation}}

\numberwithin{equation}{section}
\newtheorem{proposition}{Proposition}[section]
\newtheorem{theorem}[proposition]{Theorem}
\newtheorem{lemma}[proposition]{Lemma}

\newtheorem{remark}[proposition]{Remark}

\def\dfrac{\displaystyle\frac}

\begin{document}
 \title{\bf
  A linearized second-order scheme for nonlinear time fractional Klein-Gordon type equations}
 \author{Pin Lyu\thanks{
 Email: lyupin1991@163.com.}\qquad
 Seakweng Vong
 \thanks{Corresponding author. Email: swvong@umac.mo.}\\
{\footnotesize  \textit{Department of Mathematics, University of Macau, Avenida da Universidade, Macau, China}}}
%\author{ }

 \date{}
 \maketitle
\begin{abstract}
 We consider difference schemes for nonlinear time fractional Klein-Gordon type equations in this paper.
 A linearized scheme is proposed to solve the problem.
 As a result,  iterative method need not be employed.
 One of the main difficulties for the analysis is that certain weight averages of the approximated solutions are considered in the discretization
 and standard energy estimates cannot be applied directly.
 By introducing a new grid function, which further approximates the solution, and using ideas in some recent studies,
 we show that the method converges with second-order accuracy in time.
 \end{abstract}

\section{Introduction}

 This paper studies finite difference schemes for nonlinear fractional order Klein-Gordon type equations.
 Fractional differential equations have applications in physics, biology and petroleum industry.
 Interested reader can refer to \cite{Kilbas2006,EBarkai_Phy,Meerschaert2001_Phy,Meerschaert2002_Phy,West2007_Bio,MFitt2009_Oill} for more details.
 One of the key features of fractional derivatives are their nonlocal dependence
 which make fractional differential equations suitable to model some phenomena.
 However, the nonlocal dependence causes difficulty in the study of these equations.
 In the past decade, many works have been done on the study of effective numerical method for time fractional differential equations, the most popular approach are finite difference, spectral and finite element, see \cite{YusteSIAM2005, VongAML,
 %SunANM2006,
 SunJCP2014,LiuSIAM2008,LiuFCAA2013,XuJCP2007,JinSIAM2013,
 JinSIAM2016,DengJCP2007,DengSIAM2008,CuiJCP2009,VongCMA2014,Zhao,Lei} and the references therein. Klein-Gordon equation is a basic equation to describe many phenomena.
 %in quantum field,
 Solving it in numerically is an interesting topic. Many efficient methods have been employed to solve the linear and nonlinear Klein-Gordon, or sin-Gordon equations successfully. Such as the Adomian's decomposition method \cite{Kaya2004,El-Sayed2003}, the variational iteration method (VIM) \cite{Batiha2007}, the He's variational iteration method \cite{Yusufo2008} and the Homotopy analysis method (HAM) \cite{Jafari2009}, and so on. When studying this kind of equations with fractional order derivative, it would be more challenging.

 In this paper, we consider finite difference schemes for nonlinear time fractional Klein-Gordon type equations
 with the following form:
\begin{align}\label{eq1}
&^C_0D_t^{\alpha} u(x,t)=\frac{\partial^2 u(x,t)}{\partial x^2}-f(u)+p(x,t), \quad  x\in (a,b),\quad t\in(0,T],\\ \label{eq2}
& u(a,t)=u(b,t)=0, \quad t\in(0,T],\\ \label{eq3}
&u(x,0)=\varphi(x), \quad u_t(x,0)=\psi(x), \quad x\in [a,b],
\end{align}
where $1<\alpha<2$, $^C_0D_t^{\alpha}$ denotes the Caputo's derivative which is defined by
$$^C_0D_t^{\alpha} u(x,t)=\frac1{\Gamma(2-\alpha)}\int_0^t\frac{\partial^2 u(x,s)}{\partial s^2}\frac{ds}{{(t-s)}^{\alpha-1}},$$
and $f$ is a continuous function and satisfies the Lipschitz condition:
\begin{align}\label{Lip}
|f(\phi_1)-f(\phi_2)|\leq L|\phi_1-\phi_2|, \quad \forall \phi_1,\phi_2 \in \Omega.
\end{align}
Here $\Omega$ is a suitable domain, and $L$ is a positive constant only depends on the domain $\Omega$.

Lots of literatures have devoted to the study of time fractional Klein-Gordon (or sin-Gordon) type equations, see also \cite{Alireza2011,CuiNPDE2009,Vong2014,Vong2015,ChenH-Taiwan,Dehghan2015,Jafari2013} and references therein. The authors in \cite{Vong2014,Vong2015} proposed space compact schemes to solve the one and two dimensional time fractional Klein-Gordon type equations, respectively, and stability and convergence were analyzed by energy method. In \cite{ChenH-Taiwan}, a fully spectral scheme with finite difference discretization in time and Legendre spectral approximation in space was derived. Moreover, the meshless method based on radial basis function was used in \cite{Dehghan2015} to obtain an unconditionally stable discrete scheme for this kind of equation.  We note that the finite difference schemes (or finite difference discretization in time) proposed in \cite{Vong2014,Vong2015,ChenH-Taiwan,Dehghan2015} are nonlinear with temporal convergence order ${\cal O}(\tau^{3-\alpha})~(1<\alpha<2)$, or linear with convergence order ${\cal O}(\tau)$. These motivate us to investigate linearized and higher temporal convergence order scheme to solve the nonlinear equations. We remark that linearized scheme was shown to be very efficiently for dealing with nonlinear problems \cite{zhao_siamJSC,SunNPDE2016,WangD2014,Ran2016}.
 The main advantage is that the nonlinear term is evaluated at previous time level so that iterative method is not needed to solve the solution at the current time level, which would be more convenience and save much computational costs. However, to our knowledge,
 the idea has not been applied to construct second temporal convergence order scheme for time-fractional differential equation. Our scheme in this paper is second-order in time. The proposed method is based on the descretization given in \cite{Alikhanov} and the idea of linearized scheme.
  We further note that the discretization formula for fractional derivatives developed in \cite{Alikhanov} is not given at grid points.
 This induce some technical difficulties for shifting the evaluation of nonlinear term to previous time level.
 Inspired by some estimates in the recent works \cite{Liao0,Liao2},
 we show that our proposed scheme converges with second-order in time.

The rest of the paper is organized as follows. In section \ref{derivation}, we first give some estimates of the discretization coefficients on the fractional derivative, and using the weighted approach we derive a linearized implicit scheme for the problem \eqref{eq1}--\eqref{eq3}. The scheme is shown to be convergent with ${\cal O}(\tau^2+h^2)$ and stable by discrete energy method in section \ref{Analysis}. Spatial fourth-order compact scheme is proposed in section \ref{compactscheme}. In section \ref{Numericalexperiments}, we test some numerical examples to confirm the theoretical results. A brief conclusion is followed in the last section.

\section{Derivation of the difference scheme}\label{derivation}

\subsection{Preliminary notations and lemmas}
The following notations are needed to present our scheme. Let $\tau=\frac{T}{N}$ and $h=\frac{b-a}{M}$ be the temporal and spatial step sizes respectively, where $N$ and $M$ are some given integers. For $n=0,1,...,N$, and $i=0,1,...,M$, denote $t_n=n\tau$, $x_i=ih$, $t_{n+\theta}=(n+\theta)\tau$ for a constant $\theta\in [0,1]$, $\varphi_i=\varphi(x_i)$ and $\psi_i=\psi(x_i)$. We next introduce the grid function spaces $\mathcal{V}_h=\big\{u|u=\{u_i|0\leq i\leq M\}\mbox{ and }u_0=u_M=0\big\}$ and $\mathcal{W}_\tau=\{w^n|0\leq n\leq N\}$. For any $u,v\in \mathcal{V}_h$, we denote
$$\delta_x u_{i-\frac12}=\frac{u_i-u_{i-1}}{h}, \quad \delta_x^2 u_i=\frac{\delta_x u_{i+\frac12}-\delta_x u_{i-\frac12}}{h}=\frac{u_{i+1}-2u_i+u_{i-1}}{h^2},$$
and the inner product and norms
$$\langle u,v \rangle=h\sum_{i=1}^{M-1}u_iv_i,\quad \|u\|=\sqrt{\langle u,u\rangle},\quad |u|_1=\sqrt{h\sum_{i=1}^M\mid\delta_x u_{i-\frac12}\mid^2},\quad \|u\|_\infty=\max_{1\leq i\leq M-1}|u_i|.$$
For any $u^n\in\mathcal{W}_\tau$, we further consider
$$ \delta_t u^{n+\frac12}=\frac{u^{n+1}-u^n}{\tau},\quad \delta_{\hat t} u^{n}=\frac{u^{n+1}-u^{n-1}}{2\tau}.$$

 Discretization on the fractional derivative of our scheme based on the following lemma, which is obtained straightly by replacing the parameter $\sigma=\frac{2-\alpha}{2}$ $(0<\alpha<1)$ in Lemma 2.3 of \cite{Liao0} with $\theta=\frac{3-\alpha}{2}$ $(1<\alpha<2)$
 and, in fact,  it also can be found in Lemma 2.3 of \cite{Liao2}.
\begin{lemma}\label{lemma_time}
 Suppose $1<\alpha<2,~\theta=\frac{3-\alpha}{2},~v(t)\in{\cal C}^2[0,T]\cap{\cal C}^3(0,T]$, and there exists a positive constant $C$ such that $v'''(t)\leq C t^{\alpha-2}$ in $[0,T]$. Then
 \begin{align*}
 &\frac1{\Gamma(2-\alpha)}\int_0^{t_{n+\theta}}\frac{v'(s)}{(t_{n+\theta}-s)^{\alpha-1}}ds
 -\Delta_t^\alpha v(t_{n+\theta})={\cal O}(\tau^2),
 \end{align*}
 where $\Delta_t^\alpha v(t_{n+\theta})=\frac{\tau^{1-\alpha}}{\Gamma(3-\alpha)}\sum_{s=0}^{n}c_{n-s}^{(n+1)}[v(t_{s+1})-v(t_s)]$, and $c_0^{(1)}=a_{0}$ for $n=0$,
 $$c_m^{(n+1)}=\left\{\begin{array}{ll}
 a_{0}+b_{1}, &m=0,\\
 a_m+b_{m+1}-b_m, &1\leq m\leq n-1,\\
 a_{n}-b_{n}, &m=n;\end{array}\right.$$
 for $n\geq1$,
 in which $ a_{0}=\theta^{2-\alpha}$, $ a_{l}=(l+\theta)^{2-\alpha}-(l-1+\theta)^{2-\alpha}$, for $l\geq1$; and $b_{l}=\frac1{3-\alpha}[(l+\theta)^{3-\alpha}-(l-1+\theta)^{3-\alpha}]-\frac12[(l+\theta)^{2-\alpha}+(l-1+\theta)^{2-\alpha}].$
 \end{lemma}

\begin{lemma}\label{akbk}
The coefficients $a_k$, $b_k$, and $c_k^{(n+1)}$ defined in Lemma \ref{lemma_time} satisfy
\begin{align*}
& (a)\quad 0<b_k<\frac{\alpha-1}{2(3-\alpha)}a_k<\frac12 a_k,~k\geq 1, \quad \quad (b)\quad \sum_{k=1}^{n}b_k<\frac12\sum_{k=0}^na_k=\frac12(n+\theta)^{2-\alpha},~n\geq 1,\\
& (c)\quad c_n^{(n+1)}>\frac{2-\alpha}{2(n+\theta)^{\alpha-1}}, \quad\quad (d)\quad c_0^{(n+1)}>c_1^{(n+1)}>\cdots>c_{n-1}^{(n+1)}>c_{n}^{(n+1)},\\
&(e)\quad (2\theta-1)c_0^{(n+1)}-\theta c_1^{(n+1)}>0,\quad \quad (f)\quad\sum_{k=0}^{n}c_k^{(n+1)}=\sum_{k=0}^{n}c_k^{(k+1)}+\sum_{k=1}^{n}b_k=(n+\theta)^{2-\alpha},\\
&(g)\quad c_{n-1}^{(n+1)}=c_{n-1}^{(n)}+b_n,~n\geq 1;\quad\quad c_k^{(n+1)}=c_k^{(n)},~0\leq k\leq n-2,~n\geq 2,\\
&(h)\quad \sum_{k=0}^n\frac1{c_k^{(n+1)}}<\frac{2{(n+1)}^\alpha}{2-\alpha},~n\geq 1.
\end{align*}
\end{lemma}
\begin{proof}
The inequalities $(a)$--$(g)$ are obtained directly by replacing $\sigma=\frac{2-\alpha}{2}$ $(0<\alpha<1)$ in Lemma 2.1 and Lemma 2.2 of \cite{Liao0} with $\theta=\frac{3-\alpha}{2}$ $(1<\alpha<2)$, which also can be found in Lemma 2.1 and Lemma 2.2 of \cite{Liao2}. Using $(c)$--$(d)$, we have
$$\sum_{k=0}^n\frac1{c_k^{(n+1)}}<\sum_{k=0}^n\frac1{c_n^{(n+1)}}<\frac{2(n+1){(n+\theta)}^{\alpha-1}}{2-\alpha}
<\frac{2{(n+1)}^\alpha}{2-\alpha},$$
thus $(h)$ is verified.
\end{proof}

Denote $\mu=\tau^{\alpha-1} \Gamma(3-\alpha)$  and
\begin{equation}\label{dkn}
d_0^{(1)}=\frac{c_0^{(1)}}{\mu}=\frac{\theta^{2-\alpha}\tau^{1-\alpha}}{\Gamma(3-\alpha)}; \quad\quad d_k^{(n+1)}=\left\{\begin{array}{ll}
\frac{c_k^{(n+1)}}{\mu}, \quad &0\leq k\leq n-1,\\
\frac{c_n^{(n+1)}-\frac{\theta}{1-\theta}b_n}{\mu}, &k=n,
\end{array}\right. \quad n\geq 1.
\end{equation}
We further have the following lemma.
\begin{lemma}\label{dk}
The above coefficients $d_{k}^{(n+1)}$ $(0\leq k\leq n\leq N-1,~n\geq 1)$ satisfy
\begin{align*}
&(a)\quad \frac{(2-\alpha)^2t_{n+\theta}^{1-\alpha}}{\Gamma(4-\alpha)}<d_n^{(n+1)}<\frac{t_{n-1+\theta}^{1-\alpha}}{\Gamma(2-\alpha)},      \quad \quad (b) \quad d_0^{(n+1)}>d_1^{(n+1)}>\cdots>d_{n-1}^{(n+1)}>d_{n}^{(n+1)}, \\
&(c) \quad(2\theta-1)d_0^{(n+1)}-\theta d_1^{(n+1)}>0,\quad \quad
(d) \quad\tau \sum_{k=0}^{n}d_k^{(n+1)}<\frac{t_{n+\theta}^{2-\alpha}}{\Gamma(3-\alpha)},\\
 &(e) \quad \tau \sum_{k=0}^{n}d_k^{(k+1)}<\frac{t_{n+\theta}^{2-\alpha}}{\Gamma(3-\alpha)},
 \quad\quad (f) \quad \tau\sum_{k=0}^n \frac1{d_k^{(k+1)}}< \frac{\Gamma(4-\alpha)T^{\alpha}}{(2-\alpha)^2} .
\end{align*}
\end{lemma}
\begin{proof}
%The proof of (a)--(e) can be found in \cite{Liao1}, here we review it briefly.
Applying Lemma \ref{akbk}(a), we get
\begin{align*}
&d_n^{(n+1)}=\frac{a_n-\frac1{1-\theta}b_n}{\mu}< \frac{a_n}{\mu}=\frac{2-\alpha}{\mu}\int_0^1\frac{ds}{(n+\theta-s)^{\alpha-1}}<\frac{t_{n-1+\theta}^{1-\alpha}}{\Gamma(2-\alpha)},\\
&d_n^{(n+1)}=\frac{a_n-\frac1{1-\theta}b_n}{\mu}>\frac{2-\alpha}{(3-\alpha)\mu}a_n>
\frac{(2-\alpha)^2}{(3-\alpha)\mu}\int_0^1\frac{ds}{(n+\theta-s)^{\alpha-1}}>\frac{(2-\alpha)^2t_{n+\theta}^{1-\alpha}}{\Gamma(4-\alpha)}.
\end{align*}
So (a) is verified. Since $b_n>0$, it follows from Lemma \ref{akbk}(d) that $c_0^{(n+1)}>c_1^{(n+1)}>\cdots>c_{n-1}^{(n+1)}>c_{n}^{(n+1)}-b_n$, then the definition \eqref{dkn} yields the inequality (b). Similarly, Lemma \ref{akbk}(e) implies that
$$(2\theta-1)d_0^{(n+1)}-\theta d_1^{(n+1)}=\big[(2\theta-1)c_0^{(n+1)}-\theta c_1^{(n+1)} \big]/\mu>0,~n\geq 2;$$
and $(2\theta-1)d_0^{(2)}-\theta d_1^{(2)}=\big[(2\theta-1)c_0^{(2)}-\theta c_1^{(2)} \big]/\mu>0$. Thus (c) is verified. The definition \eqref{dkn} also implies $\tau\sum_{k=0}^n d_k^{(n+1)}<\frac{\tau}{\mu}\sum_{k=0}^n c_k^{(n+1)}$ such that the inequality (d) is obtained by using Lemma \ref{akbk}(f). The proof of (a) shows that $d_k^{(k+1)}<\frac{a_k}{\mu}$ for $k\geq 1$. Then Lemma \ref{akbk}(b) yields the inequality (e).
It is easy to check that $\frac{\theta}{\Gamma(3-\alpha)}>\frac{(2-\alpha)^2}{\Gamma(4-\alpha)}$ for $1<\alpha<2$, then $d_0^{(1)}=\frac{\theta t_\theta^{1-\alpha}}{\Gamma(3-\alpha)}>\frac{(2-\alpha)^2t_\theta^{1-\alpha}}{\Gamma(4-\alpha)}$, and it follows by combining (a) that
$$\tau\sum_{k=0}^n \frac1{d_k^{(k+1)}}<\tau\sum_{k=0}^n\frac{\Gamma(4-\alpha)t_{n+\theta}^{\alpha-1}}{(2-\alpha)^2}<
\frac{\Gamma(4-\alpha)t_{n+1}^{\alpha}}{(2-\alpha)^2}\leq \frac{\Gamma(4-\alpha)T^{\alpha}}{(2-\alpha)^2} ,$$
so inequality (f) is verified.
\end{proof}
We have the following lemma relating solution values at different points.
\begin{lemma}\label{time}
For any $g(t)\in{\cal C}^2[t_{n-1+\theta},t_{n+\theta}]$, it holds that
$$g(t_n)=(1-\theta)g(t_{n+\theta})+\theta g(t_{n-1+\theta})+{\cal O}(\tau^2).$$
\end{lemma}
\begin{proof}
Using Taylor formula with integral remainder, we have
\begin{align*}
&g(t_{n+\theta})=g(t_n)+\theta\tau g'(t_n)+\theta^2\tau^2 \int_0^1 g''(t_n+\rho\theta\tau)(1-\rho)d\rho,\\
&g(t_{n-1+\theta})=g(t_n)-(1-\theta)\tau g'(t_n)+(1-\theta)^2\tau^2 \int_0^1 g''\big(t_n-\rho(1-\theta)\tau\big)(1-\rho)d\rho.
\end{align*}
Then the desired result can be obtained by a direct calculation.
\end{proof}

\subsection{Weighted approximation to time fractional derivative}
Denote $v=u_t$, then $^C_0D_t^{\alpha} u=\frac1{\Gamma(2-\alpha)}\int_0^{t}\frac{v_t(s)}{(t-s)^{\alpha-1}}ds$.
 We define the grid functions
$$U_i^n=u(x_i,t_n),\mbox{ and }  ~ V_i^{n+\sigma}=v(x_i,t_{n+\sigma})~ \mbox{ for }~ 0\leq\sigma\leq 1 ,~ 0\leq i\leq M, ~ 0\leq n\leq N.$$
Consider equation \eqref{eq1} at the point $(x_i,t_n)$, we have
\begin{align}\label{eeq1}
^C_0D_t^{\alpha} u(x_i,t_n)=\frac{\partial^2 u(x_i,t_n)}{\partial x_i^2}-f\big(u(x_i,t_n)\big)+p(x_i,t_n), \quad  1\leq i\leq M-1, ~ 1\leq n\leq N.
\end{align}
Utilizing Lemma \ref{time} and Lemma \ref{lemma_time}, we introduce a weighted approximation for the time fractional derivative. Let $^C_0D_t^{\alpha} u(x_i,t)\in{\cal C}^2[t_{n-1+\theta},t_{n+\theta}]$ and $v(x_i,t)\in{\cal C}^2[0,T]\cap{\cal C}^3(0,T]$, $1\leq n\leq N-1$, it follows that
\begin{align}\nonumber
^C_0D_t^{\alpha} u(x_i,t_n)=&(1-\theta) ^C_0D_t^{\alpha} u(x_i,t_{n+\theta})+\theta ^C_0D_t^{\alpha} u(x_i,t_{n-1+\theta})+{\cal O}(\tau^2),\\\label{Rhat}
=&(1-\theta) \Delta_t^\alpha  V_i^{n+\theta}+\theta \Delta_t^\alpha  V_i^{n-1+\theta}+({\hat R}_t)_i^n,
\end{align}
where $({\hat R}_t)_i^n={\cal O}(\tau^2)$.

%For the sake of simplification, we will omit the subscript $i$ of $V_i^{n+\sigma}$ and $U_i^n$ in somewhere of the following derivation.
Denote $V_i^{k+1-\theta}=(1-\theta)V_i^{k+1}+\theta V_i^k$ ($k\geq0$). Then, for $n\geq1$, using Lemma \ref{akbk}(g) and \eqref{dkn}, we have
\begin{align}\nonumber
(1-\theta) \Delta_t^\alpha  V_i^{n+\theta}+\theta \Delta_t^\alpha  V_i^{n-1+\theta}
=& \frac{(1-\theta)}{\mu} \sum_{k=0}^n c_{n-k}^{(n+1)}(V_i^{k+1}-V_i^k) + \frac{\theta}{\mu} \sum_{k=0}^{n-1} c_{n-k-1}^{(n)}(V_i^{k+1}-V_i^k) \\\nonumber
=& \frac{(1-\theta)}{\mu} \sum_{k=0}^n c_{n-k}^{(n+1)}(V_i^{k+1}-V_i^k) + \frac{\theta}{\mu} \sum_{k=0}^{n-1} c_{n-k-1}^{(n+1)}(V_i^{k+1}-V_i^k)\\\nonumber
&- \frac{\theta}{\mu} b_n (V_i^1-V_i^0) \\\nonumber
=& \frac{(1-\theta)}{\mu} \sum_{k=1}^n c_{n-k}^{(n+1)}(V_i^{k+1}-V_i^k) + \frac{\theta}{\mu} \sum_{k=1}^{n} c_{n-k}^{(n+1)}(V_i^k-V_i^{k-1})\\\nonumber
&+\frac{(1-\theta)}{\mu}c_n^{(n+1)}(V_i^1-V_i^0)- \frac{\theta}{\mu} b_n (V_i^1-V_i^0) \\\nonumber
=& \frac1{\mu} \sum_{k=1}^n c_{n-k}^{(n+1)} (V_i^{k+1-\theta}-V_i^{k-\theta})\\\nonumber
&+\frac1{\mu}\big[(1-\theta)c_n^{(n+1)}-\theta b_n\big](V_i^1-V_i^0)\\\label{timedisc}
=& \sum_{k=1}^n d_{n-k}^{(n+1)}(V_i^{k+1-\theta}-V_i^{k-\theta})+ d_n^{(n+1)}(V_i^{1-\theta}-V_i^0).
\end{align}
If $n=0$,
\begin{align}\label{timedisc0}
\Delta_t^\alpha V_i^\theta=\frac{d_0^{(1)}}{1-\theta}(V_i^{1-\theta}-V_i^0).
\end{align}
For $k\geq 1$, Taylor expansion gives
\begin{align}\nonumber
V_i^{k+1-\theta}=&(2-2\theta)V_i^{k+\frac12}+(2\theta-1)V_i^k+(2-2\theta)(R_{v})_i^{k+\frac12}\\\nonumber
=&(2-2\theta)\delta_tU_i^{k+\frac12}+(2\theta-1)\delta_{\hat t}U_i^{k}-(2-2\theta)(R_t)_i^{k+\frac12}-(2\theta-1)(R_{\hat t})_i^{k}\\\label{vktheta}
&+(2-2\theta)(R_{v})_i^{k+\frac12},
\end{align}
and
\begin{align}\nonumber
V_i^{1-\theta}&=(2-2\theta)V_i^{\frac12}+(2\theta-1)V_i^0+(2-2\theta)(R_{v})_i^{\frac12}\\\label{v01theta}
&=(2-2\theta)\delta_tU_i^{\frac12}+(2\theta-1)\psi_i-(2-2\theta)(R_t)_i^{\frac12}+(2-2\theta)(R_{v})_i^{\frac12},
\end{align}
where
$$(R_{v})_i^{k+\frac12}=\frac{\tau^2}{8}\int_0^1\Big[\frac{\partial^3u(x_i,t_{k+\frac12}-\frac{\rho\tau}{2})}{\partial t^3}+\frac{\partial^3u(x_i,t_{k+\frac12}+\frac{\rho\tau}{2})}{\partial t^3} \Big](1-\rho) d\rho \mbox{ for } k\geq 0,$$
$$(R_t)_i^{k+\frac12}=\frac{\tau^2}{16}\int_0^1\Big[\frac{\partial^3u(x_i,t_{k+\frac12}-\frac{\rho\tau}{2})}{\partial t^3}+\frac{\partial^3u(x_i,t_{k+\frac12}+\frac{\rho\tau}{2})}{\partial t^3} \Big](1-\rho)^2 d\rho \mbox{ for } k\geq 0,$$
and
$$(R_{\hat t})_i^{k}=\frac{\tau^2}{4}\int_0^1\Big[\frac{\partial^3 u(x_i,t_{k}-\rho\tau)}{\partial t^3}+\frac{\partial^3 u(x_i,t_k+\rho\tau)}{\partial t^3} \Big](1-\rho)^2 d\rho \mbox{ for } k\geq 1.$$
So, by inserting \eqref{vktheta} and \eqref{v01theta} into \eqref{timedisc}, we can get the following approximation for the fractional derivative on grid function $U_i^n$
\begin{align}\nonumber
&(1-\theta) \Delta_t^\alpha  V_i^{n+\theta}+\theta \Delta_t^\alpha  V_i^{n-1+\theta}\\\nonumber
=&\sum_{k=1}^n d_{n-k}^{(n+1)}(V_i^{k+1-\theta}-V_i^{k-\theta})+ d_n^{(n+1)}(V_i^{1-\theta}-V_i^0)\\\nonumber
=&\sum_{k=1}^n d_{n-k}^{(n+1)}\big[(2-2\theta)\big(\delta_tU_i^{k+\frac12}-\delta_tU_i^{k-\frac12}\big)+(2\theta-1)\big(\delta_{\hat t}U_i^{k}-\delta_{\hat t}U_i^{k-1} \big)\big]\\\label{Rtilde}
&+d_n^{(n+1)}\big[(2-2\theta)\delta_tU_i^{\frac12}+(2\theta-1)\psi_i-V_i^0 \big]-({\tilde R}_t)_i^{n+1},\quad 1\leq n\leq N-1,~1\leq i\leq M-1,
\end{align}
where
\begin{align*}
({\tilde R}_t)_i^{n+1} =&\sum_{k=1}^n d_{n-k}^{(n+1)}\big[(2-2\theta)\big((R_t)_i^{k+\frac12}-(R_t)_i^{k-\frac12}\big)+(2\theta-1)\big((R_{\hat t})_i^{k}-(R_{\hat t})_i^{k-1}\big)\\
&-(2-2\theta)\big((R_{v})_i^{k+\frac12}-(R_{v})_i^{k-\frac12}\big)\big]+(2-2\theta)d_n^{(n+1)}\big[(R_t)_i^{\frac12}-(R_{v})_i^{\frac12}\big].
\end{align*}
Utilizing the Taylor expansion, it is observed that
\begin{align*}
(R_t)_i^{k+\frac12}-(R_t)_i^{k-\frac12}=&\frac{\tau^3}{16}\int_0^1\int_0^1\frac{\partial^4 u(x_i,t_{k+\frac12}+(\varrho-\frac{\rho}{2})\tau)}{\partial t^4}(1-\rho)^2d\varrho d\rho\\
&+\frac{\tau^3}{16}\int_0^1\int_0^1\frac{\partial^4 u(x_i,t_{k+\frac12}+(\varrho+\frac{\rho}{2})\tau)}{\partial t^4}(1-\rho)^2d\varrho d\rho,
\\
(R_{\hat t})_i^{k}-(R_{\hat t})_i^{k-1}=&\frac{\tau^3}{4}\int_0^1\int_0^1\frac{\partial^4 u(x_i,t_{k-1}+(\varrho-\rho)\tau)}{\partial t^4}(1-\rho)^2d\varrho d\rho\\
&+\frac{\tau^3}{4}\int_0^1\int_0^1\frac{\partial^4 u(x_i,t_{k-1}+(\varrho+\rho)\tau)}{\partial t^4}(1-\rho)^2d\varrho d\rho,
\\
(R_{v})_i^{k+\frac12}-(R_{v})_i^{k-\frac12}=&\frac{\tau^3}{8}\int_0^1\int_0^1\frac{\partial^4 u(x_i,t_{k+\frac12}+(\varrho-\frac{\rho}{2})\tau)}{\partial t^4}(1-\rho)d\varrho d\rho\\
&+\frac{\tau^3}{8}\int_0^1\int_0^1\frac{\partial^4 u(x_i,t_{k+\frac12}+(\varrho+\frac{\rho}{2})\tau)}{\partial t^4}(1-\rho)d\varrho d\rho, \quad 1\leq k\leq n.
\end{align*}
These three equations above yield
\begin{align*}%\label{3R}
|(R_t)_i^{k+\frac12}-(R_t)_i^{k-\frac12}|\leq C_1 \tau^3,~|(R_{\hat t})_i^{k}-(R_{\hat t})_i^{k-1}|\leq C_1 \tau^3,~ |(R_{v})_i^{k+\frac12}-(R_{v})_i^{k-\frac12}|\leq C_1\tau^3,
\end{align*}
where $C_1$ is a positive constant.

Then, suppose $\frac{\partial^4 u}{\partial t^4}$ is continuous on $[a,b]\times[0,T]$, we can obtain
\begin{align*}
|({\tilde R}_t)_i^{n+1}|
 \leq & (2-2\theta)\sum_{k=1}^n d_{n-k}^{(n+1)}|(R_t)_i^{k+\frac12}-(R_t)_i^{k-\frac12}|+(2\theta-1)\sum_{k=1}^n d_{n-k}^{(n+1)}|(R_{\hat t})_i^{k}-(R_{\hat t})_i^{k-1}|\\
 & +(2-2\theta)\sum_{k=1}^n d_{n-k}^{(n+1)}|(R_{v})_i^{k+\frac12}-(R_{v})_i^{k-\frac12}|+(2-2\theta)d_n^{(n+1)}\big(|(R_t)_i^{\frac12}|+|(R_{v})_i^{\frac12}|\big)\\
 \leq & (2-2\theta)C_1\sum_{k=1}^n d_{n-k}^{(n+1)}\tau^3+(2\theta-1)C_1\sum_{k=1}^n d_{n-k}^{(n+1)}\tau^3+ C_2d_n^{(n+1)}\tau^2 \\
 \leq&(2-2\theta)\frac{C_1t_{n+\theta}^{2-\alpha}}{\Gamma(3-\alpha)}\tau^2+(2\theta-1)\frac{C_1t_{n+\theta}^{2-\alpha}}{\Gamma(3-\alpha)}\tau^2+ \frac{C_2t_{n-1+\theta}^{1-\alpha}}{\Gamma(2-\alpha)}\tau^2 \\
 \leq & C_3 \tau^2,
\end{align*}
where Lemma \ref{dk}(a),(d) have been used, and $C_2,C_3$ are positive constants.

\subsection{Approximation on first time level and space discretization}
Note that the discretization \eqref{Rtilde} is devoted to solving the numerical solutions of $U^{n+1}$ ($n\geq 1$). For the approximation on the first time level, further construction is required.
Inserting \eqref{v01theta} into \eqref{timedisc0}, we have
\begin{align}\label{Rttheta}
\Delta_t^\alpha V_i^\theta=\frac{d_0^{(1)}}{1-\theta}(2-2\theta)\big(\delta_t U_i^{\frac12}-\psi_i \big)-({\tilde R}_t)_i^\theta,
\end{align}
in which
\begin{align*}
({\tilde R}_t)_i^\theta=\frac{d_0^{(1)}}{1-\theta}(2-2\theta)\big[(R_t)_i^{\frac12}-(R_{v})_i^{\frac12}\big]
=\frac{2\theta^{2-\alpha}\tau^{1-\alpha}}{\Gamma(3-\alpha)}\big[(R_t)_i^{\frac12}-(R_{v})_i^{\frac12}\big],
\end{align*}
hence we have $|({\tilde R}_t)_i^\theta|\leq \frac{2\theta^{2-\alpha}\tau^{1-\alpha}}{\Gamma(3-\alpha)}\big(|(R_t)_i^{\frac12}|+|(R_{v})_i^{\frac12}|\big)\leq C_3\tau^{3-\alpha}$.

We can use the Taylor formula to get
\begin{align}\label{rr1}
u(x_i,t_\theta)=u(x_i,t_0)+\theta\tau u_t(x_i,t_0)+{\cal O}(\tau^2)=\varphi_i+\theta\tau\psi_i+{\cal O}(\tau^2), \quad 1\leq i\leq M-1.
\end{align}
Consider equation \eqref{eq1} at the point $(x_i,t_\theta)$, we have
\begin{align}\label{rr2}
^C_0D_t^{\alpha} u(x_i,t_\theta)=\frac{\partial^2 u(x_i,t_\theta)}{\partial x_i^2}-f\big(u(x_i,t_\theta)\big)+p(x_i,t_\theta), \quad  1\leq i\leq M-1.
\end{align}
Then combining Lemma \ref{lemma_time}, \eqref{rr1} and \eqref{rr2}, we obtain the approximation of first time level
\begin{align}\label{utheta}
\Delta_t^\alpha V_i^{\theta}=(\varphi_{xx}+\theta\tau\psi_{xx})_i- f(\varphi_i+\theta\tau\psi_i)
+p_i^\theta+({\hat R}_t)_i^\theta,\quad  1\leq i\leq M-1,
\end{align}
where $({\hat R}_t)_i^\theta={\cal O}(\tau^2)$.

For space discretization at each of grid points, suppose $\frac{\partial^4 u}{\partial x^4}$ is continuous on $[a,b]\times[0,T]$.
 The Taylor expansion gives
\begin{align} \label{rx}
^C_0D_t^{\alpha} u(x_i,t_n)=\delta_x^2U_i^n-f\big(u(x_i,t_n)\big)+p(x_i,t_n)+(R_x)_i^n, \quad  1\leq i\leq M-1, ~ 1\leq n\leq N,
\end{align}
where $(R_x)_i^n={\cal O}(h^2)$.

\subsection{The linearized scheme}
To construct a stable implicit difference scheme, we need one more Lemma on the approximation of $U_i^n$, which is obtained by using Taylor expansion. This lemma plays an important role in analysis in the next section, see the equality \eqref{ww}. It reads as:
\begin{lemma}\label{w}
For $n\geq 1$, it holds that
$$U_i^n=\frac{W_i^{n+1}+W_i^n}{2}+(R_{w})_i^n,$$
where
\begin{align*}
W_i^{n}=&(\frac32-\theta)\big[\theta U_i^{n}+(1-\theta)U_i^{n-1}\big]+(\theta-\frac12)\big[\theta
U_i^{n-1}+(1-\theta)U_i^{n-2}\big], \quad n\geq 2,
\\
W_i^{1}=&(\frac32-\theta)\big[\theta U_i^{1}+(1-\theta)U_i^{0}\big]+(\theta-\frac12)\big[\theta
U_i^{0}+(1-\theta)(U_i^1-2\tau\psi_i)\big],
\end{align*}
and $(R_{w})_i^n={\cal O}(\tau^2)$.
\end{lemma}

Therefore, by \eqref{Rhat}, \eqref{Rtilde}, \eqref{Rttheta}, \eqref{utheta}, \eqref{rx} and Lemma \ref{w}, one can approximate the equation \eqref{eq1}--\eqref{eq3} as following:
\begin{align}\nonumber
& (1-\theta) \Delta_t^\alpha  V_i^{n+\theta}+\theta \Delta_t^\alpha  V_i^{n-1+\theta}=\delta_x^2 \Big(\frac{W_i^{n+1}+W_i^n}{2}\Big)-f(U_i^n)+p_i^n+(R_x)_i^n-({\hat R}_t)_i^n+(R_{w})_i^n, \\\label{sch1}
&~~~~~~~~~~~~~~~~~~~~~~~~~~~~~~~~~~~~~~~~~~~~~~~~~~~~~~~~~~~~~~~1\leq n\leq N-1,~ 1\leq i\leq M-1,\\\label{sch2}
&\Delta_t^\alpha V_i^\theta=(\varphi_{xx}+\theta\tau\psi_{xx})_i-f(\varphi_i+\theta\tau\psi_i)
+p_i^\theta+({\hat R}_t)_i^\theta,~~~~~~~~~~ 1\leq i\leq M-1,  \\\label{sch3}
& U_0^n=U_M^n=0, \quad 1\leq n\leq N,\\\label{sch4}
& U_i^0=\varphi_i,\quad V_i^0=\psi_i,\quad  0\leq i\leq M,
\end{align}
in which
\begin{align}\nonumber
(1-\theta) \Delta_t^\alpha  V_i^{n+\theta}&+\theta \Delta_t^\alpha  V_i^{n-1+\theta}=\sum_{k=1}^n d_{n-k}^{(n+1)}\big[(2-2\theta)\big(\delta_tU_i^{k+\frac12}-\delta_tU_i^{k-\frac12}\big)\\\label{sch5}
&+(2\theta-1)\big(\delta_{\hat t}U_i^{k}-\delta_{\hat t}U_i^{k-1}\big)\big] +d_n^{(n+1)}(2-2\theta)(\delta_tU_i^{\frac12}-\psi_i)-({\tilde R}_t)_i^{n+1},
\end{align}
and
\begin{align}\label{sch6}
\Delta_t^\alpha V_i^\theta=2d_0^{(1)}\big(\delta_t U_i^{\frac12}-\psi_i \big)-({\tilde R}_t)_i^\theta.
\end{align}
Denote $u_i^n$, $v_i^n$ and $w_i^n$ the numerical approximations of $U_i^n$, $V_i^n$ and $W_i^n$, respectively. Omitting the small terms $(R_x)_i^n$, $({\hat R}_t)_i^n$, $(R_{w})_i^n$ in \eqref{sch1}, $({\hat R}_t)_i^\theta$ in \eqref{sch2}, $({\tilde R}_t)_i^{n+1}$ in \eqref{sch5} and $({\tilde R}_t)_i^\theta$ in \eqref{sch6}, we obtain the following linearized difference scheme:
\begin{align}\nonumber
& (1-\theta) \Delta_t^\alpha  v_i^{n+\theta}+\theta \Delta_t^\alpha  v_i^{n-1+\theta}=\delta_x^2 \Big(\frac{w_i^{n+1}+w_i^n}{2}\Big)-f(u_i^n)+p_i^n, \\\label{sc1}
&~~~~~~~~~~~~~~~~~~~~~~~~~~~~~~~~~~~~~~~~ 1\leq n\leq N-1,~ 1\leq i\leq M-1,\\\label{sc2}
& \Delta_t^\alpha v_i^\theta=(\varphi_{xx}+\theta\tau\psi_{xx})_i-f(\varphi_i+\theta\tau\psi_i)
+p_i^\theta, \quad 1\leq i\leq M-1,\\\label{sc3}
& u_0^n=u_M^n=0, \quad 1\leq n\leq N,\\\label{sc4}
& u_i^0=\varphi_i,\quad v_i^0=\psi_i,\quad  0\leq i\leq M.
\end{align}

\section{Analysis of the proposed scheme}\label{Analysis}

Before carrying out the convergence and stability of difference scheme \eqref{sc1}--\eqref{sc4}, we first list some preliminary lemmas.

\begin{lemma}\label{GRW}
(Gronwall's inequality \cite{Gronwall}) Let $\{G_n\}$ and $\{k_n\}$ be nonnegative sequences satisfying
$$G_0\leq K,\qquad G_n \leq K+\sum_{l=0}^{n-1}k_l G_l,\quad n\geq 1, $$
where $K\geq 0$. Then
$$G_n\leq K\exp\Big(\sum_{l=0}^{n-1}k_l \Big),\quad n\geq 1. $$
\end{lemma}

\begin{lemma}\label{1norm} (\cite{Sun1})
Let $u\in \mathcal{V}_h$, it holds that
$$ \|u\|^2\leq \frac{(b-a)^2}{6}|u|_1^2.$$
\end{lemma}

\begin{lemma}\label{alik} (\cite{Alikhanov})
If the positive sequence $\{g_k^{(n+1)}|0\leq k\leq n,n\geq 1\}$ is strictly decreasing and satisfies $(2\sigma-1)g_0^{(n+1)}-\sigma g_1^{(n+1)}>0$ for a constant $\sigma\in(0,1)$ . Then
$$
2[\sigma y^{n+1}+(1-\sigma)y^n]\sum_{k=0}^{n}g_{n-k}^{(n+1)}(y^{k+1}-y^k) \geq \sum_{k=0}^{n}g_{n-k}^{(n+1)}\big[(y^{k+1})^2-(y^k)^2 \big],\quad n\geq 1.
$$
\end{lemma}
 We need a special form of Lemma \ref{alik}, which is stated as the following:
%Then we can get the following lemma:
\begin{lemma}\label{stproof1}
For any real sequence $F^n$, the following inequality holds:
\begin{align*}
&2[\theta v^{n+1-\theta}+(1-\theta)v^{n-\theta}][(1-\theta) \Delta_t^\alpha v^{n+\theta}+\theta \Delta_t^\alpha  v^{n-1+\theta}-F^n]\\
\geq &\sum_{k=0}^n \frac{c_{n-k}^{(n+1)}}{\mu} {(v^{k+1-\theta})}^2- \sum_{k=0}^{n-1} \frac{c_{n-k-1}^{(n)}}{\mu} {(v^{k+1-\theta})}^2-\frac{ b_n}{(1-\theta)\mu}{(v^{1-\theta})}^2-d_n^{(n+1)}\Big(v^0+\frac1{d_n^{(n+1)}}F^n \Big)^2.
\end{align*}
\end{lemma}

\begin{proof}
Taking $g_k^{(n+1)}=d_k^{(n+1)}$, $y^n=v^{n-\theta}$ $(n\geq 1)$ and $y^0=v^0+\frac1{d_n^{(n+1)}}F^n$ in Lemma \ref{alik}, and using Lemma \ref{dk}(b),(c), we get
\begin{align*}
&2\big[\theta v^{n+1-\theta} +(1-\theta)v^{n-\theta} \big]\big[(1-\theta) \Delta_t^\alpha v^{n+\theta}+\theta \Delta_t^\alpha  v^{n-1+\theta}- F^n \big]\\
\geq& \sum_{k=1}^n d_{n-k}^{(n+1)} \big[ {(v^{k+1-\theta})}^2-{(v^{k-\theta})}^2\big]+d_n^{(n+1)} \Big[{(v^{1-\theta})}^2-\Big(v^0+\frac1{d_n^{(n+1)}}F^n \Big)^2\Big] \\
=& \sum_{k=0}^n \frac{c_{n-k}^{(n+1)}}{\mu} {(v^{k+1-\theta})}^2- \sum_{k=1}^n \frac{c_{n-k}^{(n+1)}}{\mu} {(v^{k-\theta})}^2 -\frac{\theta b_n}{(1-\theta)\mu}{(v^{1-\theta})}^2-d_n^{(n+1)}\Big(v^0+\frac1{d_n^{(n+1)}}F^n \Big)^2\\
=& \sum_{k=0}^n \frac{c_{n-k}^{(n+1)}}{\mu} {(v^{k+1-\theta})}^2- \sum_{k=1}^n \frac{c_{n-k}^{(n)}}{\mu} {(v^{k-\theta})}^2 -\frac{ b_n}{(1-\theta)\mu}{(v^{1-\theta})}^2-d_n^{(n+1)}\Big(v^0+\frac1{d_n^{(n+1)}}F^n \Big)^2\\
=& \sum_{k=0}^n \frac{c_{n-k}^{(n+1)}}{\mu} {(v^{k+1-\theta})}^2- \sum_{k=0}^{n-1} \frac{c_{n-k-1}^{(n)}}{\mu} {(v^{k+1-\theta})}^2 -\frac{ b_n}{(1-\theta)\mu}{(v^{1-\theta})}^2-d_n^{(n+1)}\Big(v^0+\frac1{d_n^{(n+1)}}F^n \Big)^2.
\end{align*}
\end{proof}

\subsection{Convergence}
Now we denote the error $e_i^n=U_i^n-u_i^n$, $0\leq i\leq M$, $0\leq n\leq N$,
and
$$(1-\theta) \Delta_t^\alpha {\hat v}_i^{n+\theta}+\theta \Delta_t^\alpha  {\hat v}_i^{n-1+\theta}=\sum_{k=1}^n d_{n-k}^{(n+1)}({\hat v}_i^{k+1-\theta}-{\hat v}_i^{k-\theta})+ d_n^{(n+1)}{\hat v}_i^{1-\theta},$$
in which
\begin{align}\nonumber
{\hat v}_i^{k+1-\theta}
=&(2-2\theta)\frac{e_i^{k+1}-e_i^{k}}{\tau}+(2\theta-1)\frac{e_i^{k+1}-e_i^{k-1}}{2\tau},\quad k\geq 1,
\\\label{v1theta}
{\hat v}_i^{1-\theta}=&(2-2\theta)\frac{e_i^{1}}{\tau}.
\end{align}
 Taking
\begin{align*}
{\hat w}_i^k=&(\frac32-\theta)\big[\theta e_i^{k}+(1-\theta)e_i^{k-1}\big]+(\theta-\frac12)\big[\theta
e_i^{k-1}+(1-\theta)e_i^{k-2}\big],\quad k\geq 2,\\
{\hat w}_i^1=&[(\frac32-\theta)\theta+(\theta-\frac12)(1-\theta)] e_i^{1},
\end{align*}
 and subtracting \eqref{sc1}--\eqref{sc4} from \eqref{sch1}--\eqref{sch4}, we obtain the following error system:
\begin{align}\nonumber
&(1-\theta) \Delta_t^\alpha {\hat v}_i^{n+\theta}+\theta \Delta_t^\alpha  {\hat v}_i^{n-1+\theta}=\delta_x^2\Big(\frac{{\hat w}_i^{n+1}+{\hat w}_i^n}{2}\Big) -\big[f(U_i^n)-f(u_i^n)\big] +R_i^{n+1}, \\\label{err1}
&~~~~~~~~~~~~~~~~~~~~~~~~~~~~~~~~~~~~~~~~~~~~~~~~~~~~\qquad 1\leq n\leq N-1,~ 1\leq i\leq M-1,\\\label{err2}
&\frac{d_0^{(1)}}{1-\theta} {\hat v}_i^{1-\theta}=R_i^1,\quad 1\leq i\leq M-1,\\\label{err3}
&e_0^n=e_M^n=0, \quad 1\leq n\leq N,\\\label{err4}
& e^0_i=0,\quad 0\leq i\leq M,
\end{align}
where $R_i^1=({\hat R}_t)_i^\theta+({\tilde R}_t)_i^\theta$ and $R_i^{n+1}=(R_x)_i^n-({\hat R}_t)_i^n+(R_{w})_i^n+({\tilde R}_t)_i^{n+1}$, $ 1\leq n\leq N-1$.
Then there exists a positive constant $C_4$ such that
\begin{align}\label{Rn}
\|R^1\|\leq C_4\tau^{3-\alpha}, \quad |R^1|_1\leq C_4\tau^{3-\alpha}\mbox{  and  }\|R^{n+1}\|\leq C_4(\tau^2+h^2), ~ 1\leq n\leq N-1.
\end{align}
Then we conclude the convergence of proposed scheme \eqref{sc1}--\eqref{sc4} as following.
\begin{theorem}\label{convergence}
  Let $u(x,t)$ be the solution of the problem \eqref{eq1}--\eqref{eq3} and smooth enough, and let $\{u_i^n,0\leq i\leq M,0\leq n\leq N\}$ be the solution of the scheme \eqref{sc1}--\eqref{sc4}. %If
  %Assume that $\tau=\nu h^{\frac12+\epsilon}$, where $\nu$, $\epsilon$ are positive numbers.
  %${\bar C}(\nu^2h^{2\epsilon}+h)\leq C_0$, in which $C_0$ is a positive constant independent of $\tau$ and $h$,
  %\textcolor{blue}{$f(u)$ satisfies \eqref{Lip}}, then
  It holds that
\begin{align}\label{converesult}
\|e^n\|\leq {\bar C}(\tau^2+h^2),\quad 0\leq n\leq N,
\end{align}
where ${\bar C}=\frac{\exp(C_7)}{2(1-\theta)}\Big[ \frac12 \Big(\frac{(\theta-\frac12)C_4\Gamma(3-\alpha)}{\theta^{2-\alpha}}\Big)^2+4\Gamma(2-\alpha)T^\alpha\Big(C_5+C_6
+\frac{2\Gamma(4-\alpha)T^{\alpha}}{(2-\alpha)^2}C_4^2\Big) \Big]^{\frac12}$, with \\
$C_5=\Big[\frac{(3\theta-2\theta^2-\frac12)C_3\Gamma(3-\alpha)}{2\theta^{2-\alpha}}\Big]^2$ , $C_6=\frac{(1-\theta)\Gamma(3-\alpha)(3-2\theta)T^{2-\alpha}}{2\theta^{4-2\alpha}}$,
and
$C_7=\frac{\Gamma(2-\alpha)\Gamma(4-\alpha)T^{2\alpha} L^2 }{\left[(2-\alpha)(1-\theta)\right]^2}$.
\end{theorem}
\begin{proof}
We have $\|e^0\|=0$ from \eqref{err4}.
Now we use mathematical induction to prove
\begin{align}\nonumber
\|e^{n}\|^2\leq & \frac{1}{4(1-\theta)^2}\left[ \frac12 \Big(\frac{(\theta-\frac12)C_4\Gamma(3-\alpha)}{\theta^{2-\alpha}}\Big)^2\right.\\\nonumber
&\left.+4\Gamma(2-\alpha)T^\alpha\Big(C_5+C_6+\frac{2\Gamma(4-\alpha)T^{\alpha}}{(2-\alpha)^2}C_4^2\Big) \right](\tau^2+h^2)^2\\\label{induction}
 &+ \frac{2 \Gamma(2-\alpha)T^\alpha L^2}{(1-\theta)^2}\tau\sum_{k=0}^{n-1} \frac1{d_k^{(k+1)}}\|e^{k}\|^2, \quad 1\leq n\leq N.
\end{align}
It follows from \eqref{v1theta}, \eqref{err2} and \eqref{Rn} that
\begin{align}\label{e1}
\|e^1\|=\frac{\tau}{2d_0^{(1)}}\|R^1\|=\frac{\Gamma(3-\alpha)\tau^\alpha}{2\theta^{2-\alpha}}\|R^1\|\leq \frac{C_4\Gamma(3-\alpha)}{2\theta^{2-\alpha}}\tau^3.
\end{align}
Hence \eqref{induction} holds for $n=1$.

Suppose \eqref{induction} is valid for $1\leq n\leq m$ ($1\leq m\leq N-1$),
we then prove it is also valid for $n=m+1$.

Taking the inner product of \eqref{err1} with
\begin{align}\label{ww}
 2\big[\theta {\hat v}_i^{n+1-\theta} +(1-\theta){\hat v}_i^{n-\theta}\big]=2\Big(\frac{{\hat w}_i^{n+1}-{\hat w}_i^n}{\tau}\Big),
 \quad 1\leq n\leq m,
\end{align}
we have
\begin{align}\nonumber
2\big\langle (1-\theta) \Delta_t^\alpha {\hat v}^{n+\theta}+\theta \Delta_t^\alpha  {\hat v}^{n-1+\theta}-R_f^{n+1},\theta {\hat v}^{n+1-\theta} +(1-\theta){\hat v}^{n-\theta}\big\rangle\\\label{error1_1}
=2\big\langle \delta_x^2 \big( \frac{{\hat w}^{n+1}+{\hat w}^n}{2} \big),\frac{{\hat w}^{n+1}-{\hat w}^n}{\tau} \big\rangle  ,
\end{align}
where $(R_f)_i^{n+1}=-\big[f(U_i^n)-f(u_i^n)\big]+R_i^{n+1}$.

It is easy to verify that
\begin{align}\label{wn1norm}
-2\Big\langle \delta_x^2 \Big(\frac{{\hat w}^{n+1}+{\hat w}^n}{2}\Big),\frac{{\hat w}^{n+1}-{\hat w}^n}{\tau}
\Big\rangle=\frac{|{\hat w}^{n+1}|_1^2-|{\hat w}^{n}|_1^2}{\tau}.
\end{align}
Noticing ${\hat v}^0=0$ and utilizing Lemma \ref{stproof1}, we get
\begin{align}\nonumber
&2\big\langle (1-\theta) \Delta_t^\alpha {\hat v}^{n+\theta}+\theta \Delta_t^\alpha  {\hat v}^{n-1+\theta}-(R_f)^{n+1},\theta {\hat v}^{n+1-\theta} +(1-\theta){\hat v}^{n-\theta}\big\rangle\\\label{wn1norm2}
\geq & \sum_{k=0}^n \frac{c_{n-k}^{(n+1)}}{\mu} \|{\hat v}^{k+1-\theta}\|^2- \sum_{k=0}^{n-1} \frac{c_{n-k-1}^{(n)}}{\mu} \|{\hat v}^{k+1-\theta}\|^2 -\frac{ b_n}{(1-\theta)\mu}\|{\hat v}^{1-\theta}\|^2-d_n^{(n+1)}\|\frac1{d_n^{(n+1)}}(R_f)^{n+1}\|^2.
\end{align}
Substituting \eqref{wn1norm} and \eqref{wn1norm2} into \eqref{error1_1}, we obtain
\begin{align}\label{error2}
 E^{n+1}-E^n \leq \frac{ \tau b_n}{(1-\theta)\mu}\|{\hat v}^{1-\theta}\|^2+\tau \frac1{d_n^{(n+1)}}\|(R_f)^{n+1}\|^2,\quad
1\leq n\leq m,
\end{align}
where
$$E^n=\tau \sum_{k=0}^{n-1} \frac{c_{n-k-1}^{(n)}}{\mu} \|{\hat v}^{k+1-\theta}\|^2+|{\hat w}^{n}|_1^2.$$
Summing up \eqref{error2} for $n$ from $1$ to $m$ yield
\begin{align*}%\label{error3}
E^{m+1} \leq  E^{1}+\frac{\tau}{(1-\theta)\mu} \sum_{n=1}^m b_n \|{\hat v}^{1-\theta}\|^2+ \tau\sum_{n=1}^m \frac1{d_n^{(n+1)}}  \|(R_f)^{n+1}\|^2,
\end{align*}
then we can deduce the following inequality
\begin{align}\label{sumck1}
\frac{\tau}{\mu}\sum_{k=0}^{m}c_{m-k}^{(m+1)}\|{\hat v}^{k+1-\theta}\|^2
\leq  |{\hat w}^{1}|_1^2+
\frac{\tau}{\mu}\Big[c_0^{(1)}+\frac1{1-\theta} \sum_{n=1}^m b_n\Big]\|{\hat v}^{1-\theta}\|^2+
\tau\sum_{n=1}^m \frac1{d_n^{(n+1)}}\|(R_f)^{n+1}\|^2.
\end{align}
It can be verified by using Cauchy-Schwarz inequality and Lemma \ref{akbk}(h) that
\begin{align}\label{sumck2}
\|\tau\sum_{k=0}^{m} {\hat v}^{k+1-\theta}\|^2\leq\left(\tau\sum_{k=0}^{m}\frac{\mu}{c_{m-k}^{(m+1)}} \right)\frac{\tau}{\mu}\sum_{k=0}^{m}c_{m-k}^{(m+1)}\|{\hat v}^{k+1-\theta}\|^2\leq 2\Gamma(2-\alpha)t_{m+1}^\alpha\left(\frac{\tau}{\mu}\sum_{k=0}^{m}c_{m-k}^{(m+1)}{\|{\hat v}^{k+1-\theta}\|}^2\right).
\end{align}
Furthermore, the inequality $(y+z)^2\leq 2(y^2+z^2)$ gives
\begin{align}\label{sumck3}
\|(\frac{3}{2}-\theta)e^{m+1}+(\theta-\frac12)e^m\|^2 = \|(\theta-\frac12)e^1+\tau\sum_{k=0}^{m} {\hat v}^{k+1-\theta}\|^2 \leq
2\left(\|(\theta-\frac12)e^1\|^2+ \|\tau\sum_{k=0}^{m} {\hat v}^{k+1-\theta}\|^2\right).
\end{align}
Consequently, it follows from \eqref{sumck1}--\eqref{sumck3} that
\begin{align}\label{errorB}
\|(\frac{3}{2}-\theta)e^{m+1}+(\theta-\frac12)e^m\|^2\leq B^m,
\end{align}
where $t_{m+1}^\alpha\leq T^\alpha$ has been used, and
\begin{align*}
B^m=&2\|(\theta-\frac12)e^1\|^2\\
&+4\Gamma(2-\alpha)T^\alpha\left\{|{\hat w}^1|_1^2+\frac{\tau}{\mu}\Big[c_0^{(1)}+\frac1{1-\theta} \sum_{n=1}^m b_n\Big]\|{\hat v}^{1-\theta}\|^2+
\tau\sum_{n=1}^m \frac1{d_n^{(n+1)}}\|(R_f)^{n+1}\|^2\right\}.
\end{align*}
We note that if $\|e^{m+1}\|\leq \|e^{m}\|$, \eqref{induction} follows directly.
Therefore, we only consider the situation that
$$\|e^{m+1}\|\geq \|e^{m}\|.$$
 Then the triangular property of $L_2$ norm yields
\begin{align*}
\|(\frac{3}{2}-\theta)e^{m+1}+(\theta-\frac12)e^m\|
\ge\big(\frac{3}{2}-\theta\big)\|e^{m+1}\|-\big(\theta-\frac12\big)\|e^m\|\geq 2(1-\theta)\|e^{m+1}\|,
\end{align*}
which implies that
\begin{align}\label{error4}
\|\big(\frac{3}{2}-\theta\big)e^{m+1}+\big(\theta-\frac12\big)e^m\|^2\ge 4(1-\theta)^2\|e^{m+1}\|^2.
\end{align}
Combining \eqref{errorB} and \eqref{error4}, we get
\begin{align}\label{error5}
\|e^{m+1}\|^2\leq \frac{B^m}{4(1-\theta)^2}.
\end{align}
Then we estimate $B^m$ term by term.
Recalling the definition of ${\hat w}_i^1$,
%and noticing $R_i^1$ is a truncation error just respects to time direction,
a straightforward calculation shows
\begin{align}\label{errorw1}
|{\hat w}^1|_1^2&=(3\theta-2\theta^2-\frac12)^2|e^1|_1^2= \Big[\frac{(3\theta-2\theta^2-\frac12)\Gamma(3-\alpha)}{2\theta^{2-\alpha}}\Big]^2 \tau^{2\alpha}|R^1|_1^2
\leq C_5\tau^4.
\end{align}
By using Lemma \ref{lemma_time}, Lemma \ref{akbk}(b), \eqref{err2} and \eqref{Rn}, we have
\begin{align}\nonumber
\frac{\tau}{\mu}\Big[c_0^{(1)}+\frac1{1-\theta} \sum_{n=1}^m b_n\Big]\|{\hat v}^{1-\theta}\|^2
\leq & \Big[\frac{{t_\theta}^{2-\alpha}}{\Gamma(3-\alpha)}+\frac{t_{n+\theta}^{2-\alpha}}{2(1-\theta)\Gamma(3-\alpha)} \Big]\|{\hat v}^{1-\theta}\|^2\\\nonumber
\leq & \frac{2(1-\theta){t_\theta}^{2-\alpha}+t_{n+\theta}^{2-\alpha}}{2(1-\theta)\Gamma(3-\alpha)} \Big(\frac{1-\theta}{d_0^{(1)}}\Big)^2\|R^1\|^2\\\nonumber
\leq & \frac{(1-\theta)\Gamma(3-\alpha)(3-2\theta)T^{2-\alpha}}{2\theta^{4-2\alpha}}\tau^{2\alpha-2}\|R^1\|^2\\\label{vhattheta}
\leq & C_6 \tau^4.
\end{align}
For the nonlinear term, assuming that the global Lipschitz condition \eqref{Lip} hold, we have
\begin{align}\label{nonlinearterm}
\|f(U^n)-f(u^n)\|\leq L\|U^n-u^n\|=L\|e^n\|.
\end{align}
Then utilizing Lemma \ref{dk}(f), \eqref{Rn} and \eqref{nonlinearterm}, we can conclude that
\begin{align}\nonumber
\tau\sum_{n=1}^m \frac1{d_n^{(n+1)}}\|(R_f)^{n+1}\|^2\leq& 2\tau\sum_{n=1}^m \frac1{d_n^{(n+1)}}\left( \|f(U^n)-f(u^n)\|^2+\|R^{n+1}\|^2\right) \\\label{Rm}
\leq &2L^2\tau\sum_{n=1}^m \frac1{d_n^{(n+1)}}\|e^n\|^2+ \frac{2\Gamma(4-\alpha)T^{\alpha}C_4^2}{(2-\alpha)^2}(\tau^2+h^2)^2.
\end{align}
Thus, \eqref{Rn}, \eqref{e1}, \eqref{error5}--\eqref{vhattheta} and \eqref{Rm} yield %\eqref{for-Gronwall} for $n=m+1$.
\begin{align*}
\|e^{m+1}\|^2\leq & \frac{1}{4(1-\theta)^2}\left[ \frac12 \Big(\frac{(\theta-\frac12)C_4\Gamma(3-\alpha)}{\theta^{2-\alpha}}\Big)^2\right.\\
&\left.+4\Gamma(2-\alpha)T^\alpha\Big(C_5+C_6+\frac{2\Gamma(4-\alpha)T^{\alpha}}{(2-\alpha)^2}C_4^2\Big) \right](\tau^2+h^2)^2\\
 &+ \frac{2 \Gamma(2-\alpha)T^\alpha L^2}{(1-\theta)^2}\tau\sum_{n=1}^m \frac1{d_n^{(n+1)}}\|e^n\|^2,
\end{align*}
which shows that \eqref{induction} is proved.

Consequently, we can apply Lemma \ref{GRW} and Lemma \ref{dk}(f) on \eqref{induction} to conclude
$$\|e^{n}\|^2\leq \big[{\bar C}(\tau^2+h^2)\big]^2, \quad 1\leq n\leq N.$$
\end{proof}
\begin{remark}\label{condition}
For functions which are not globally Lipschitz continuous, for example $f(u)=[u(x,t)]^r$, $r$ is positive integer and $r\geq 3$.
 Inspired by the approach for dealing with the nonlinear term in Theorem 4.2 of \cite{zhao_siamJSC},
 one can also obtain convergence of the scheme by assuming $\tau=\nu h^{\frac12+\epsilon}$,  where $\nu$, $\epsilon$ are positive numbers.
 In fact, based on the smoothness assumption of the exact solution, there exists a positive constant $C_0$ such that
\begin{align*}%\label{Uassumeption}
\|U^n\|_\infty\leq C_0, \quad 0\leq n\leq N.
\end{align*}
Note that \eqref{converesult} is valid for $1\leq n\leq m$ by applying Lemma \ref{GRW} and Lemma \ref{dk}(f) on the inductive assumption of Theorem \ref{convergence}. If ${\bar C}(\nu^2h^{2\epsilon}+h)\leq {\tilde C}_0$, in which ${\tilde C}_0$ is a positive constant independent of $\tau$ and $h$,  then it follows that
$$\|e^n\|_\infty\leq h^{-1}\|e^n\|\leq {\bar C}(\nu^2h^{2\epsilon}+h)\leq {\tilde C}_0, \quad 1\leq n\leq m,$$
and we get
\begin{equation}\label{infinity-bound}
\|u^n\|_\infty=\|u^n-U^n+U^n\|_\infty\leq \|e^n\|_\infty+\|U^n\|_\infty\leq {\tilde C}_0+C_0,\quad 1\leq n\leq m,
\end{equation}
which implies that $u^n$ ($1\leq n\leq m$) is uniformly bounded.
  The bound \eqref{nonlinearterm} can still be obtained so long as  $f$ is Lipschitz continuous on $[-{\tilde C}_0-C_0,{\tilde C}_0+C_0]$.
One can then follow the other parts of the proof to conclude convergence of the scheme.
\end{remark}

\subsection{Stability}
Now we show the stability of propose scheme \eqref{sc1}--\eqref{sc4}. Suppose that $\{{\tilde u}_i^n,0\leq i\leq M,0\leq n\leq N\}$ is the solution of the following difference scheme:
\begin{align}\nonumber
& (1-\theta) \Delta_t^\alpha  {\tilde v}_i^{n+\theta}+\theta \Delta_t^\alpha  {\tilde v}_i^{n-1+\theta}=\delta_x^2 \Big(\frac{{\tilde w}_i^{n+1}+{\tilde w}_i^n}{2}\Big)-f({\tilde u}_i^n) +p_i^n, \\\label{stb1}
&~~~~~~~~~~~~~~~~~~~~~~~~~~~~~~~~~~~~~~~~ 1\leq n\leq N-1,~ 1\leq i\leq M-1,\\\label{stb2}
& \Delta_t^\alpha {\tilde v}_i^\theta=({\tilde \varphi}_{xx}+\theta\tau{\tilde \psi}_{xx})_i-f({\tilde \varphi}_i+\theta\tau{\tilde \psi}_i) +p_i^\theta, \quad 1\leq i\leq M-1,\\\label{stb3}
& {\tilde u}_i^0={\tilde \varphi}_i, ~ ({\tilde u}_t)_i^0={\tilde \psi}_i, \quad  0\leq i\leq M,\\\label{stb4}
& {\tilde u}_0^n={\tilde u}_M^n=0, \quad 1\leq n\leq N,
\end{align}
where
$$(1-\theta) \Delta_t^\alpha  {\tilde v}_i^{n+\theta}+\theta \Delta_t^\alpha  {\tilde v}_i^{n-1+\theta}=\sum_{k=1}^n d_{n-k}^{(n+1)}({\tilde v}_i^{k+1-\theta}-{\tilde v}_i^{k-\theta})+ d_n^{(n+1)}({\tilde v}_i^{1-\theta}-{\tilde \psi}_i),$$
with
\begin{align*}
{\tilde v}_i^{k+1-\theta}=&(2-2\theta)\delta_t {\tilde u}_i^{k+\frac12}+(2\theta-1)\delta_{\hat t} {\tilde u}_i^{ k},\quad k\geq 1,\\
{\tilde v}_i^{1-\theta}=&(2-2\theta)\delta_t {\tilde u}_i^{\frac12}+(2\theta-1){\tilde \psi}_i;
\end{align*}
and
\begin{align*}
{\tilde w}_i^{k}=&(\frac32-\theta)\big[\theta {\tilde u}_i^{k}+(1-\theta){\tilde u}_i^{k-1}\big]+(\theta-\frac12)\big[\theta
{\tilde u}_i^{k-1}+(1-\theta){\tilde u}_i^{k-2}\big], \quad k\geq 2,\\
{\tilde w}_i^{1}=&(\frac32-\theta)\big[\theta {\tilde u}_i^{1}+(1-\theta){\tilde u}_i^{0}\big]+(\theta-\frac12)\big[\theta
{\tilde u}_i^{0}+(1-\theta)({\tilde u}_i^1-2\tau{\tilde \psi}_i)\big].
\end{align*}
Denoting the perturbation term
$$\eta_i^n={\tilde u}_i^n-u_i^n, \quad 1\leq n\leq N,~ 1\leq i\leq M-1,$$
and taking $\xi_i=\big[({\tilde \varphi}_{xx}-\varphi_{xx})_i+\theta\tau({\tilde \psi}_{xx}-\psi_{xx})_i\big]-\big[f({\tilde \varphi}_i+\theta\tau{\tilde \psi}_i)-f(\varphi_i+\theta\tau\psi_i)\big]$.

Then we can get the following perturbation system:
\begin{align}\nonumber
&(1-\theta) \Delta_t^\alpha {\hat \eta}_i^{n+\theta}+\theta \Delta_t^\alpha  {\hat \eta}_i^{n-1+\theta}=\delta_x^2\Big(\frac{{\hat \zeta}_i^{n+1}+{\hat \zeta}_i^n}{2}\Big) -\big[f({\tilde u}_i^n)-f(u_i^n)\big], \\\label{per1}
&~~~~~~~~~~~~~~~~~~~~~~~~~~~~~~~~~~~~~~~~~~\quad 1\leq n\leq N-1,~ 1\leq i\leq M-1,\\\label{per2}
&\frac{d_0^{(1)}}{1-\theta} {\hat \eta}_i^{1-\theta}=\xi_i,\quad 1\leq i\leq M-1,\\\label{per3}
&\eta_0^n=\eta_M^n=0, \quad 1\leq n\leq N,\\\label{per4}
& \eta^0_i={\tilde \varphi}_i-\varphi_i,\quad 0\leq i\leq M.
\end{align}
where
$$(1-\theta) \Delta_t^\alpha {\hat \eta}_i^{n+\theta}+\theta \Delta_t^\alpha  {\hat \eta}_i^{n-1+\theta}=\sum_{k=1}^n d_{n-k}^{(n+1)}({\hat \eta}_i^{k+1-\theta}-{\hat \eta}_i^{k-\theta})+ d_n^{(n+1)}{\hat \eta}_i^{1-\theta},$$
with
\begin{align}\nonumber
{\hat \eta}_i^{k+1-\theta}
=&(2-2\theta)\frac{\eta_i^{k+1}-\eta_i^{k}}{\tau}+(2\theta-1)\frac{\eta_i^{k+1}-\eta_i^{k-1}}{2\tau},\quad k\geq 1,\\\label{eta1theta}
{\hat \eta}_i^{1-\theta}=&(2-2\theta)\frac{\eta_i^{1}-\eta_i^0}{\tau}+(2\theta-1)({\tilde \psi}_i-\psi_i),
\end{align}
and
\begin{align}\nonumber
{\hat \zeta}_i^k=&(\frac32-\theta)\big[\theta \eta_i^{k}+(1-\theta)\eta_i^{k-1}\big]+(\theta-\frac12)\big[\theta
\eta_i^{k-1}+(1-\theta)\eta_i^{k-2}\big],\quad k\geq 2,\\\label{stable1_1}
{\hat \zeta}_i^1=&(\frac32-\theta)\big[\theta \eta_i^1+(1-\theta)\eta_i^0\big]+(\theta-\frac12)\big[\theta
\eta_i^0+(1-\theta)\big(\eta_i^1-2\tau({\tilde \psi}_i-\psi_i)\big)\big].
\end{align}
We have the following theorem to describe the stability of proposed scheme.
\begin{theorem}\label{perstable}
Let $\{\eta_i^n,0\leq i\leq M,0\leq n\leq N\}$ be the solution of the perturbation system \eqref{per1}--\eqref{per4}.  %If \textcolor{blue}{$f(u)$ satisfies \eqref{Lip}},
It holds that
\begin{align}\label{perstable1}
\|\eta^n\|^2\leq {\tilde C}^2\Big(C_8(|{\tilde \varphi}_{xx}- {\varphi}_{xx}|_1^2+|{\tilde \psi}_{xx}-{\psi}_{xx}|_1^2)+C_9|\eta^0|_1^2+C_{10}|{\tilde \psi}-\psi|_1^2 \Big), \quad 0\leq n\leq N,
\end{align}
where ${\tilde C}=\exp(C_7)$, $C_8=\frac1{4(1-\theta)^2}\Big\{\Big(\frac{(\theta-\frac12)(b-a)\Gamma(3-\alpha)}{\theta^{2-\alpha}}\Big)^2+\\ 12\Gamma(2-\alpha)T^\alpha\Big[\Big(\frac{(3\theta-2\theta^2-\frac12)\Gamma(3-\alpha)}{\theta^{2-\alpha}}\Big)^2
+\frac{2(b-a)^2C_5}{9}\Big]\Big\}$, \\ $C_9=\frac1{4(1-\theta)^2}\Big\{{[(2\theta-1)^2(L+1)^2+(\frac32-\theta)^2](b-a)^2}+ \\ 12\Gamma(2-\alpha)T^\alpha\Big[{(3\theta-2\theta^2-\frac12)^2(L^2+1)
+(\frac32-3\theta+2\theta^2)^2}+\frac{2(b-a)^2 C_5L^2}{9}\Big]\Big\}$ and $C_{10}=\frac{1}{4(1-\theta)^2}\Big\{(2\theta-1)^2\Big[6\Big(\frac{(2\theta-1)(b-a)}{2\sqrt{6}(1-\theta)}+L \Big)^2+(b-a)^2 \Big]
+12\Gamma(2-\alpha)T^\alpha\Big[\big(\frac{(2\theta-1)(3\theta-2\theta^2-\frac12)}{2-2\theta}\big)^2(L^2+1)+(3\theta-2\theta^2-1)^2+\frac{2(b-a)^2 C_5L^2}{9}+\frac{(b-a)^2 T^{2-\alpha}}{9\Gamma(3-\alpha)}\Big]\Big\}$.
\end{theorem}
\begin{proof}
Obviously, \eqref{perstable1} is valid for $n=0$.
We use mathematical induction once again to prove
\begin{align}\nonumber
\|\eta^{n}\|^2\leq &C_8(|{\tilde \varphi}_{xx}- {\varphi}_{xx}|_1^2+|{\tilde \psi}_{xx}-{\psi}_{xx}|_1^2)+C_9|\eta^0|_1^2+C_{10}|{\tilde \psi}-\psi|_1^2\\\label{induc_stable}
&+\frac{2\Gamma(2-\alpha)T^{\alpha}L^2}{(1-\theta)^2}\tau\sum_{k=0}^{n-1}\frac1{d_k^{(k+1)}} \|\eta^k\|^2,\quad
1\leq n\leq N.
\end{align}
It follows from \eqref{per2}, \eqref{eta1theta}, and Lemma \ref{1norm} that
\begin{align}\nonumber
\|\eta^1\|&=\|\frac{\tau}{2d_0^{(1)}}\xi-\frac{(2\theta-1)\tau}{2-2\theta}({\tilde \psi-\psi})+\eta^0\|\leq \frac{\tau}{2d_0^{(1)}}\|\xi\|+\frac{(2\theta-1)\tau}{2-2\theta}\|{\tilde \psi}-\psi\|+\|\eta^0\|\\\label{stable1_21}
&\leq \frac{(b-a)\Gamma(3-\alpha)\tau^\alpha}{2\sqrt{6}\theta^{2-\alpha}}|\xi|_1+\frac{(2\theta-1)(b-a)\tau}{2\sqrt{6}(1-\theta)}|{\tilde \psi}-\psi|_1+\frac{b-a}{\sqrt{6}}|\eta^0|_1 .
\end{align}
Note that
$$|f({\tilde \varphi}_i+\theta\tau{\tilde \psi}_i)-f(\varphi_i+\theta\tau\psi_i)|\leq L|{\tilde \varphi}_i-\varphi_i+\theta\tau({\tilde \psi}_i-\psi_i)|\leq L(|\eta_i^0|+\theta\tau|{\tilde \psi}_i-\psi_i|),$$
and then (for sufficiently small $\tau$)
\begin{align}\label{xi}
|\xi|_1\leq |{\tilde \varphi}_{xx}- {\varphi}_{xx}|_1+|{\tilde \psi}_{xx}-{\psi}_{xx}|_1+L(|\eta^0|_1+\tau|{\tilde \psi}-\psi|_1).
\end{align}
Combining \eqref{stable1_21} and \eqref{xi}, we get
\begin{align}\nonumber
\|\eta^1\|\leq& \frac{(b-a)\Gamma(3-\alpha)\tau^\alpha}{2\sqrt{6}\theta^{2-\alpha}}(|{\tilde \varphi}_{xx}- {\varphi}_{xx}|_1+|{\tilde \psi}_{xx}-{\psi}_{xx}|_1)+\Big(\frac{(2\theta-1)(b-a)}{2\sqrt{6}(1-\theta)}+L\Big)\tau|{\tilde \psi}-\psi|_1\\\label{stable1_2}
&+\frac{b-a}{\sqrt{6}}(L+1)|\eta^0|_1 .
\end{align}
Since $\sqrt{C_9}>\frac{\sqrt{6}(b-a)}{3}(L+1)$, then \eqref{induc_stable} is valid for $n=1$.

Suppose \eqref{induc_stable} hold for $1\leq n\leq m$ ($1\leq m\leq N-1$). Now we show \eqref{induc_stable} also hold for $n=m+1$.

Taking the inner product of \eqref{per1} by
$$ 2\big[\theta {\hat \eta}_i^{n+1-\theta} +(1-\theta){\hat \eta}_i^{n-\theta}\big]=2\Big(\frac{{\hat \zeta}_i^{n+1}-{\hat \zeta}_i^n}{\tau}\Big), \quad 1\leq n\leq m,$$
we get
$$2\big\langle (1-\theta) \Delta_t^\alpha {\hat \eta}^{n+\theta}+\theta \Delta_t^\alpha  {\hat \eta}^{n-1+\theta}-{\tilde R}_f^{n+1},\theta {\hat \eta}^{n+1-\theta} +(1-\theta){\hat \eta}^{n-\theta}\big\rangle=2\big\langle \delta_x^2 \big( \frac{{\hat \zeta}^{n+1}+{\hat \zeta}^n}{2} \big),\frac{{\hat \zeta}^{n+1}-{\hat \zeta}^n}{\tau} \big\rangle  ,$$
where $({\tilde R}_f)_i^{n+1}=-\big[f({\tilde u}_i^n)-f(u_i^n)\big]$.

Then following the similar methodology in the proof of Theorem \ref{convergence}, we can obtain
\begin{align}\label{etan}
\|\eta^{m+1}\|^2\leq \frac{{\tilde B}^m}{4(1-\theta)^2},
\end{align}
where
\begin{align}\nonumber
{\tilde B}^m
=&6\big(\|(\theta-\frac12)\eta^1\|^2+\|(\frac32-\theta)\eta^0\|^2+\|\tau(2\theta-1)({\tilde\psi}-\psi)\|^2\big)+4\Gamma(2-\alpha)T^\alpha\Big\{|{\hat \zeta}^1|_1^2\\\label{Btilde}
&+\frac{\tau}{\mu}\Big[c_0^{(1)}+\frac1{(1-\theta)} \sum_{n=1}^m b_n\Big]{\|{\hat \eta}^{1-\theta}\|}^2+
\tau\sum_{n=1}^m d_n^{(n+1)}\big\|({\tilde\psi}-\psi)+\frac1{d_n^{(n+1)}}({\tilde R}_f)^{n+1}\big\|^2\Big\}.
\end{align}
Applying \eqref{stable1_1}, \eqref{eta1theta}, \eqref{xi} and Cauchy-Schwarz inequality, we have
\begin{align}\nonumber
|{\hat \zeta}^1|_1^2\leq & \Big[(3\theta-2\theta^2-\frac12)|\eta^1|_1+(\frac32-3\theta+2\theta^2)|\eta^0|_1+(3\theta-2\theta^2-1)\tau|{\tilde \psi}-\psi|_1 \Big]^2\\\nonumber
\leq & 3\Big\{\Big[\frac{(3\theta-2\theta^2-\frac12)\Gamma(3-\alpha)}{\theta^{2-\alpha}}\Big]^2{\tau^{2\alpha}}(|{\tilde \varphi}_{xx}- {\varphi}_{xx}|_1^2+|{\tilde \psi}_{xx}-{\psi}_{xx}|_1^2) \\\nonumber &+\Big[(3\theta-2\theta^2-\frac12)^2(L^2+1)+(\frac32-3\theta+2\theta^2)^2\Big]|\eta^0|_1^2\\\label{zeta1}
&+\Big[\Big(\frac{(2\theta-1)(3\theta-2\theta^2-\frac12)}{2-2\theta}\Big)^2(L^2+1)+(3\theta-2\theta^2-1)^2\Big]\tau^2|{\tilde \psi}-\psi|_1^2\Big\}.
\end{align}
Note that
\begin{align}\label{etahat}
\frac{\tau}{\mu}\Big[c_0^{(1)}+\frac1{(1-\theta)} \sum_{n=1}^m b_n\Big]{\|{\hat \eta}^{1-\theta}\|}^2\leq C_5\|\xi\|^2\leq \frac{C_5(b-a)^2}{6}|\xi|_1^2,
\end{align}
and the globally Lipschitz continuity \eqref{Lip} yields
$$\|f({\tilde u}^n)-f(u^n)\|\leq L \|{\tilde u}^n-u^n\|=L\|\eta^n\|,$$
and then
\begin{align}\label{etahat}
\tau\sum_{n=1}^m d_n^{(n+1)}\big\|({\tilde\psi}-\psi)+\frac1{d_n^{(n+1)}}({\tilde R}_f)^{n+1}\big\|^2\leq & \frac{(b-a)^2 T^{2-\alpha}}{3\Gamma(3-\alpha)}|{\tilde\psi}-\psi|_1^2+2L^2\tau\sum_{n=1}^m\frac1{d_n^{(n+1)}} \|\eta^n\|^2.
\end{align}
So we can derive form combining \eqref{xi}--\eqref{etahat} and Lemma \ref{1norm} that
{\small{\begin{align*}
\|\eta^{m+1}\|^2&\leq C_8(|{\tilde \varphi}_{xx}- {\varphi}_{xx}|_1^2+|{\tilde \psi}_{xx}-{\psi}_{xx}|_1^2)+C_9|\eta^0|_1^2+C_{10}|{\tilde \psi}-\psi|_1^2+\frac{2\Gamma(2-\alpha)T^{\alpha}L^2}{(1-\theta)^2}\tau\sum_{n=1}^m\frac1{d_n^{(n+1)}} \|\eta^n\|^2.\\
&\leq C_8(|{\tilde \varphi}_{xx}- {\varphi}_{xx}|_1^2+|{\tilde \psi}_{xx}-{\psi}_{xx}|_1^2)+C_9|\eta^0|_1^2+C_{10}|{\tilde \psi}-\psi|_1^2+\frac{2\Gamma(2-\alpha)T^{\alpha}L^2}{(1-\theta)^2}\tau\sum_{n=0}^m\frac1{d_n^{(n+1)}} \|\eta^n\|^2,
\end{align*}}}
which shows that \eqref{induc_stable} is proved.

Therefore, applying Lemma \ref{GRW} and Lemma \ref{dk}(f) on \eqref{induc_stable}, we finally get
{\small{\begin{align*}
\|\eta^{n}\|^2\leq \exp(2C_7)\Big(C_8(|{\tilde \varphi}_{xx}- {\varphi}_{xx}|_1^2+|{\tilde \psi}_{xx}-{\psi}_{xx}|_1^2)+C_9|\eta^0|_1^2+C_{10}|{\tilde \psi}-\psi|_1^2 \Big),~ 1\leq n\leq N.
\end{align*}}}
\end{proof}

\begin{remark}
For nonlinear terms which are locally Lipschitz continuous as discussed in Remark \ref{condition}, we assume that
$$|\eta^0|_1\leq C_{11} h^{1+\delta},\quad |{\tilde \psi}-\psi|_1\leq C_{11}h^{1+\delta},\quad |{\tilde \varphi}_{xx}- {\varphi}_{xx}|_1\leq C_{11} h^{1+\delta},\quad|{\tilde \psi}_{xx}-{\psi}_{xx}|_1\leq C_{11} h^{1+\delta},$$
where $C_{11}$ and $\delta$ are positive constants.
Note that \eqref{perstable1} is valid for $1\leq n\leq m$ by applying Lemma \ref{GRW} and Lemma \ref{dk}(f) on \eqref{induc_stable}.
If ${\tilde C}C_{11}\sqrt{C_8+C_9+C_{10}}h^\delta\leq {\tilde C}_0$, it follows that
$$\|\eta^n\|_\infty\leq h^{-1}\|\eta^n\|\leq {\tilde C}C_{11}\sqrt{C_8+C_9+C_{10}}h^\delta\leq {\tilde C}_0,\quad 1\leq n\leq m.$$
Noticing \eqref{infinity-bound},
then
$$\|{\tilde u}^n\|_\infty=\|{\tilde u}^n-u^n+u^n\|_\infty\leq \|\eta^n\|_\infty+\|u^n\|_\infty\leq 2{\tilde C}_0+C_0,\quad 1\leq n\leq m,$$
which implies that ${\tilde u}^n$ is uniformly bounded.
If  $f$ is Lipschitz continuous on an interval containing $[-2{\tilde C}_0-C_0,2{\tilde C}_0+C_0]$,
one can still obtain the desired conclusion.
\end{remark}

\section{Compact scheme}\label{compactscheme}
In this section, we propose the spacial fourth-order scheme for the problem \eqref{eq1}--\eqref{eq3}.
For any $u\in \mathcal{V}_h$, we define the spatial high order operator ${\cal A}$ as follow:
$${\cal A}u_i=\left\{\begin{array}{ll}
 \frac1{12}(u_{i-1}+10u_i+u_{i+1}), &1\leq i\leq M-1,\\
 u_i, &i=0,~M,\end{array}\right.$$
and we define the corresponding norm:
$$\|u\|_A=\sqrt{\langle {\cal A}u,u\rangle}.$$
Then it is easy to check that
\begin{align}\label{uAnorm}
\frac23\|u\|\leq \|u\|_A\leq \|u\|.
\end{align}
By Taylor expansion, if $\frac{\partial^6 u}{\partial x^6}$ is continuous on $[x_{i-1},x_{i+1}]$, we have
\begin{align}\label{compactA}
{\cal A}\frac{\partial^2 u}{\partial x^2}(x_i)=\delta_x^2 u(x_i)+O(h^4).
\end{align}
Performing the compact operator ${\cal A}$ on both sides of \eqref{eeq1} and following the derivation of the difference scheme \eqref{sc1}--\eqref{sc4} in section \ref{derivation}, we obtain the compact difference linearized scheme for \eqref{eq1}--\eqref{eq3}:
\begin{align}\nonumber
& (1-\theta) \Delta_t^\alpha {\cal A} v_i^{n+\theta}+\theta \Delta_t^\alpha {\cal A} v_i^{n-1+\theta}=\delta_x^2 \big(\frac{w_i^{n+1}+w_i^n}{2}\big)-{\cal A}f(u_i^n) +{\cal A}p_i^n, \\\label{Csc1}
&~~~~~~~~~~~~~~~~~~~~~~~~~~~~~~~~~~~~~~~~~~~~~~~~~ 1\leq n\leq N-1,~ 1\leq i\leq M-1,\\\label{Csc2}
& \Delta_t^\alpha v_i^\theta=(\varphi_{xx}+\theta\tau\psi_{xx})_i-f(\varphi_i+\theta\tau\psi_i)
+p_i^\theta, \quad 1\leq i\leq M-1,\\\label{Csc3}
& u_0^n=u_M^n=0, \quad 1\leq n\leq N,\\\label{Csc4}
& u_i^0=\varphi_i,\quad v_i^0=\psi_i,\quad  0\leq i\leq M.
\end{align}
Using similar approach with the proof of Theorem 3.7 in \cite{Vong_numer} or Lemma 4.2 in \cite{Liao2}, we have the following Lemma.
\begin{lemma}\label{compactvtheta}
For any real sequence $F^n$, the following estimate holds:
\begin{align*}
&2\big\langle\theta v^{n+1-\theta} +(1-\theta)v^{n-\theta} ,(1-\theta) \Delta_t^\alpha {\cal A}v^{n+\theta}+\theta \Delta_t^\alpha {\cal A} v^{n-1+\theta}- {\cal A}F^n \big\rangle\\
\geq& \sum_{k=1}^n d_{n-k}^{(n+1)} \big( \|v^{k+1-\theta}\|_A^2-\|v^{k-\theta}\|_A^2\big)+d_n^{(n+1)}\big( \|v^{1-\theta}\|_A^2-\|v^0+\frac1{d_n^{(n+1)}}F^n\|_A^2\big)\\
=& \sum_{k=0}^n \frac{c_{n-k}^{(n+1)}}{\mu} \|v^{k+1-\theta}\|_A^2- \sum_{k=0}^{n-1} \frac{c_{n-k-1}^{(n)}}{\mu} \|v^{k+1-\theta}\|_A^2 -\frac{ b_n}{(1-\theta)\mu}\|v^{1-\theta}\|_A^2-d_n^{(n+1)}\|v^0+\frac1{d_n^{(n+1)}}F^n \|_A^2.
\end{align*}
\end{lemma}
Then, with the help of \eqref{uAnorm}, \eqref{compactA} and Lemma \ref{compactvtheta}, and under the same assumptions in Theorem \ref{convergence} and Theorem \ref{perstable}, the theoretical results can be obtained following  similar arguments in the proof of Theorem \ref{convergence} and Theorem \ref{perstable}.
 We present the convergence conclusion in the following.
\begin{theorem}\label{compconverge}
Let $u(x,t)$ be the solution of the problem \eqref{eq1}--\eqref{eq3} and smooth enough, and let $\{u_i^n,0\leq i\leq M,0\leq n\leq N\}$ be the solution of the scheme \eqref{Csc1}--\eqref{Csc4}. %If ${\hat C}(\nu^2h^{2\epsilon}+h^3)\leq C_0$,
 If \eqref{Lip} holds globally, we then have
\begin{align}\label{compconveresult}
\|e^n\|\leq {\hat C}(\tau^2+h^4),\quad 1\leq n\leq N,
\end{align}
where ${\hat C}=\frac{3\exp(\frac94C_7)}{4(1-\theta)}\Big[ \frac12 \Big(\frac{(\theta-\frac12)C_4\Gamma(3-\alpha)}{\theta^{2-\alpha}}\Big)^2+4\Gamma(2-\alpha)T^\alpha\Big(C_5+C_6
+\frac{2\Gamma(4-\alpha)T^{\alpha}}{(2-\alpha)^2}C_4^2\Big) \Big]^{\frac12}$.
\end{theorem}
\begin{remark}
As in Remark \ref{condition}, we can show that \eqref{compconveresult} holds for locally Lipschitz continuous nonlinear term,
if we assume ${\hat C}(\nu^2h^{2\epsilon}+h^3)\leq {\tilde C}_0$.
\end{remark}
\section{Numerical experiments}\label{Numericalexperiments}
 In this section, we carry out numerical experiments for the proposed finite difference schemes \eqref{sc1}--\eqref{sc4} and \eqref{Csc1}--\eqref{Csc4} to illustrate our theoretical statements. All our tests were done in MATLAB R2014a with a desktop computer (Dell optiplex 7020, configuration: Intel(R) Core(TM) i7-4790 CPU 3.60GHz and 16.00G RAM).

 The $L_2$ norm errors between the exact and the numerical solutions
 $$E_{2}(\tau,h)=\max_{0\leq n\leq N}\|e^n\|$$
 are shown in the following tables.
 In the tables,
 $$\mbox{Rate1}=\log_2\bigg(\dfrac{E_{2}(2\tau,h)}{E_{2}(\tau,h)}\bigg)$$
 is used to denote the temporal convergence order for sufficiently small $h$, and
 $$\mbox{Rate2}=\log_2\bigg(\dfrac{E_{2}(\tau,2h)}{E_{2}(\tau,h)}\bigg)$$
 is  the spatial convergence order for sufficiently small $\tau$.

\subsection{Accuracy verification}\label{Accuracy}
 We consider the problem \eqref{eq1}--\eqref{eq3} for $x\in[0,1]$, $T=1$ and the forcing term
 $$p(x,t)=\left[\frac{24}{\Gamma(5-\alpha)}t^{4-\alpha}+\pi^2(t^4+1) \right]\sin(\pi x)+f\big(u(x,t)\big),$$
is chosen to such that the exact solution is $u(x,t)=\sin(\pi x)(t^4+1)$, where
 \begin{align*}
 &\mbox{\bf Case 1}\quad f\big(u(x,t)\big)=2\left(u(x,t)\right)^3 ,\\
 &\mbox{\bf Case 2}\quad f\big(u(x,t)\big)=\sin\big(u(x,t)\big), \quad\mbox{(sin-Gordon)} ,\\
 &\mbox{\bf Case 3}\quad f\big(u(x,t)\big)=\left[\big(u(x,t)\big)^2+5\right]^{\frac12}.
 \end{align*}
The numerical results for the above three cases by applying difference scheme \eqref{sc1}--\eqref{sc4} were recorded in Table \ref{table1} and Table \ref{table2}, while the results for the three cases by applying compact scheme \eqref{Csc1}--\eqref{Csc4} were presented in Table \ref{table3} and Table \ref{table4}.

  \begin{table}[hbt!]
 \begin{center}
 \caption{Numerical accuracy in temporal direction of scheme \eqref{sc1}--\eqref{sc4} with $h=\frac1{1000}$. }\label{table1}
 \renewcommand{\arraystretch}{1}
 \def\temptablewidth{1\textwidth}
 {\rule{\temptablewidth}{0.9pt}}
 \begin{tabular*}{\temptablewidth}{@{\extracolsep{\fill}}cccccccc}
   &$\tau$   &\multicolumn{2}{c}{$\alpha=1.2$}&\multicolumn{2}{c}{$\alpha=1.5$}
              &\multicolumn{2}{c}{$\alpha=1.8$}\\
         \cline{3-4}             \cline{5-6}             \cline{7-8}\\
             &&$E_2(\tau,h)$& Rate1   &$E_2(\tau,h)$& Rate1   &$E_2(\tau,h)$& Rate1  \\\hline
             &$1/20$     & 2.5994e-03  & $\ast$  & 3.0095e-03  & $\ast$  & 3.0680e-03  & $\ast$ \\
 {\bf Case 1}&$1/40$     & 6.5070e-04  & 1.9981  & 7.5142e-04  & 2.0018  & 7.6428e-04  & 2.0051 \\
             &$1/80$     & 1.6242e-04  & 2.0023  & 1.8761e-04  & 2.0019  & 1.9053e-04  & 2.0041 \\
             &$1/160$    & 4.0295e-05  & 2.0110  & 4.6593e-05  & 2.0095  & 4.7242e-05  & 2.0119 \\
 %$1/320$    & 9.7507e-06  & 2.0470  & 1.1329e-05  & 2.0401  & 1.1454e-05  & 2.0442 \\
     \hline
             &$1/20$     & 5.4877e-03  & $\ast$  & 5.5724e-03  & $\ast$  & 4.7836e-03  & $\ast$ \\
 {\bf Case 2}&$1/40$     & 1.3773e-03  & 1.9944  & 1.3994e-03  & 1.9935  & 1.1950e-03  & 2.0011 \\
             &$1/80$     & 3.4422e-04  & 2.0004  & 3.4977e-04  & 2.0003  & 2.9756e-04  & 2.0058 \\
             &$1/160$    & 8.5409e-05  & 2.0109  & 8.6803e-05  & 2.0106  & 7.3449e-05  & 2.0184 \\
                  \hline
             &$1/20$     & 5.0042e-03  & $\ast$  & 5.2116e-03  & $\ast$  & 4.5829e-03  & $\ast$ \\
 {\bf Case 3}&$1/40$     & 1.2539e-03  & 1.9967  & 1.3072e-03  & 1.9953  & 1.1442e-03  & 2.0019 \\
             &$1/80$     & 3.1327e-04  & 2.0010  & 3.2666e-04  & 2.0006  & 2.8493e-04  & 2.0057 \\
             &$1/160$    & 7.7724e-05  & 2.0109  & 8.1076e-05  & 2.0104  & 7.0359e-05  & 2.0178
\end{tabular*}
{\rule{\temptablewidth}{0.9pt}}
\end{center}
\end{table}
\begin{table}[hbt!]
\begin{center}
 \caption{Numerical accuracy in spatial direction of scheme \eqref{sc1}--\eqref{sc4} with $\tau=\frac1{1000}$. }\label{table2}
 \renewcommand{\arraystretch}{1}
 \def\temptablewidth{1\textwidth}
 {\rule{\temptablewidth}{0.9pt}}
 \begin{tabular*}{\temptablewidth}{@{\extracolsep{\fill}}cccccccc}
   &$h$   &\multicolumn{2}{c}{$\alpha=1.2$}&\multicolumn{2}{c}{$\alpha=1.5$}
              &\multicolumn{2}{c}{$\alpha=1.8$}\\
         \cline{3-4}             \cline{5-6}             \cline{7-8}\\
             &&$E_2(\tau,h)$& Rate2   &$E_2(\tau,h)$& Rate2   &$E_2(\tau,h)$& Rate2  \\\hline
             &$1/20$     & 1.0942e-03  & $\ast$  & 1.3052e-03  & $\ast$  & 1.6671e-03  & $\ast$ \\
 {\bf Case 1}&$1/40$     & 2.7284e-04  & 2.0038  & 3.2598e-04  & 2.0014  & 4.1632e-04  & 2.0016 \\
             &$1/80$     & 6.7455e-05  & 2.0160  & 8.1203e-05  & 2.0052  & 1.0363e-04  & 2.0062 \\
             &$1/160$    & 1.6611e-05  & 2.0218  & 2.0018e-05  & 2.0202  & 2.5468e-05  & 2.0247 \\
     \hline
             &$1/20$     & 2.3688e-03  & $\ast$  & 2.2720e-03  & $\ast$  & 2.6671e-03  & $\ast$ \\
 {\bf Case 2}&$1/40$     & 5.9008e-04  & 2.0051  & 5.6571e-04  & 2.0058  & 6.6450e-04  & 2.0049 \\
             &$1/80$     & 1.4583e-04  & 2.0166  & 1.3971e-04  & 2.0177  & 1.6465e-04  & 2.0129 \\
             &$1/160$    & 3.4798e-05  & 2.0673  & 3.3241e-05  & 2.0714  & 3.9735e-05  & 2.0509 \\
                  \hline
             &$1/20$     & 2.1390e-03  & $\ast$  & 2.0746e-03  & $\ast$  & 2.4736e-03  & $\ast$ \\
 {\bf Case 3}&$1/40$     & 5.3295e-04  & 2.0048  & 5.1668e-04  & 2.0055  & 6.1637e-04  & 2.0047 \\
             &$1/80$     & 1.3171e-04  & 2.0166  & 1.2757e-04  & 2.0180  & 1.5269e-04  & 2.0132 \\
             &$1/160$    & 3.1417e-05  & 2.0677  & 3.0319e-05  & 2.0730  & 3.6806e-05  & 2.0526
\end{tabular*}
{\rule{\temptablewidth}{0.9pt}}
\end{center}
\end{table}

Table \ref{table1} reports the numerical results in time direction of the proposed scheme \eqref{sc1}--\eqref{sc4} with fixed $h=\frac1{1000}$ and different choices of $\alpha$ are taken. Meanwhile,  the numerical results in space direction with fixed $\tau=\frac1{1000}$ and different $\alpha$ are listed in Table \ref{table2}. From these two tables, one can see that the convergence rate for the three cases are both 2 in time and space, which are in accordance with the theoretical statements.

  \begin{table}[hbt!]
 \begin{center}
 \caption{Numerical accuracy in temporal direction of scheme \eqref{Csc1}--\eqref{Csc4} with $h=\frac1{100}$. }\label{table3}
 \renewcommand{\arraystretch}{1}
 \def\temptablewidth{1\textwidth}
 {\rule{\temptablewidth}{0.9pt}}
 \begin{tabular*}{\temptablewidth}{@{\extracolsep{\fill}}cccccccc}
  &$\tau$   &\multicolumn{2}{c}{$\alpha=1.2$}&\multicolumn{2}{c}{$\alpha=1.5$}
              &\multicolumn{2}{c}{$\alpha=1.8$}\\
         \cline{3-4}             \cline{5-6}             \cline{7-8}\\
                        &&$E_2(\tau,h)$& Rate1   &$E_2(\tau,h)$& Rate1   &$E_2(\tau,h)$& Rate1  \\\hline
             &$1/20$     & 2.5998e-03  & $\ast$  & 3.0099e-03  & $\ast$  & 3.0685e-03  & $\ast$ \\
 {\bf Case 1}&$1/40$     & 6.5113e-04  & 1.9974  & 7.5184e-04  & 2.0012  & 7.6475e-04  & 2.0045 \\
             &$1/80$     & 1.6285e-04  & 1.9994  & 1.8803e-04  & 1.9995  & 1.9099e-04  & 2.0015 \\
             &$1/160$    & 4.0729e-05  & 1.9995  & 4.7017e-05  & 1.9997  & 4.7701e-05  & 2.0014 \\
      %$1/320$    & 1.0183e-05  & 1.9999  & 1.1753e-05  & 2.0002  & 1.1911e-05  & 2.0017 \\
          \hline
             &$1/20$     & 5.4886e-03  & $\ast$  & 5.5733e-03  & $\ast$  & 4.7847e-03  & $\ast$ \\
 {\bf Case 2}&$1/40$     & 1.3782e-03  & 1.9936  & 1.4003e-03  & 1.9928  & 1.1961e-03  & 2.0001 \\
             &$1/80$     & 3.4517e-04  & 1.9975  & 3.5068e-04  & 1.9975  & 2.9863e-04  & 2.0019 \\
             &$1/160$    & 8.6352e-05  & 1.9990  & 8.7708e-05  & 1.9994  & 7.4512e-05  & 2.0028 \\
           \hline
             &$1/20$     & 5.0051e-03  & $\ast$  & 5.2124e-03  & $\ast$  & 4.5839e-03  & $\ast$ \\
 {\bf Case 3}&$1/40$     & 1.2548e-03  & 1.9960  & 1.3080e-03  & 1.9946  & 1.1452e-03  & 2.0010 \\
             &$1/80$     & 3.1412e-04  & 1.9981  & 3.2749e-04  & 1.9979  & 2.8591e-04  & 2.0020 \\
             &$1/160$    & 7.8576e-05  & 1.9991  & 8.1903e-05  & 1.9995  & 7.1346e-05  & 2.0027
\end{tabular*}
{\rule{\temptablewidth}{0.9pt}}
\end{center}
\end{table}
  \begin{table}[hbt!]
\begin{center}
 \caption{Numerical accuracy in spatial direction of scheme \eqref{Csc1}--\eqref{Csc4} with $\tau=\frac1{5000}$. }\label{table4}
 \renewcommand{\arraystretch}{1}
 \def\temptablewidth{1\textwidth}
 {\rule{\temptablewidth}{0.9pt}}
 \begin{tabular*}{\temptablewidth}{@{\extracolsep{\fill}}cccccccc}
  &$h$   &\multicolumn{2}{c}{$\alpha=1.2$}&\multicolumn{2}{c}{$\alpha=1.5$}
              &\multicolumn{2}{c}{$\alpha=1.8$}\\
         \cline{3-4}             \cline{5-6}             \cline{7-8}\\
                        &&$E_2(\tau,h)$& Rate2   &$E_2(\tau,h)$& Rate2   &$E_2(\tau,h)$& Rate2  \\\hline
             &$1/4$     & 8.6534e-04  & $\ast$  & 1.0312e-03  & $\ast$  & 1.3172e-03  & $\ast$ \\
 {\bf Case 1}&$1/8$     & 5.3074e-05  & 4.0272  & 6.3283e-05  & 4.0263  & 8.0835e-05  & 4.0263 \\
             &$1/16$    & 3.2641e-06  & 4.0232  & 3.9227e-06  & 4.0119  & 5.0070e-06  & 4.0130 \\
             &$1/32$    & 1.9472e-07  & 4.0672  & 2.3157e-07  & 4.0823  & 2.9156e-07  & 4.1021 \\
          \hline
             &$1/4$     & 1.8717e-03  & $\ast$  & 1.7946e-03  & $\ast$  & 2.1058e-03  & $\ast$ \\
 {\bf Case 2}&$1/8$     & 1.1476e-04  & 4.0277  & 1.1000e-04  & 4.0281  & 1.2907e-04  & 4.0281 \\
             &$1/16$    & 7.0565e-06  & 4.0235  & 6.7590e-06  & 4.0245  & 7.9587e-06  & 4.0195 \\
             &$1/32$    & 3.5757e-07  & 4.3026  & 3.3780e-07  & 4.3226  & 4.2591e-07  & 4.2239 \\
           \hline
             &$1/4$     & 1.6891e-03  & $\ast$  & 1.6377e-03  & $\ast$  & 1.9516e-03  & $\ast$ \\
 {\bf Case 3}&$1/8$     & 1.0365e-04  & 4.0264  & 1.0047e-04  & 4.0268  & 1.1973e-04  & 4.0268 \\
             &$1/16$    & 6.3734e-06  & 4.0235  & 6.1724e-06  & 4.0248  & 7.3808e-06  & 4.0198 \\
             &$1/32$    & 3.2240e-07  & 4.3051  & 3.1045e-07  & 4.3134  & 4.0458e-07  & 4.1893
\end{tabular*}
{\rule{\temptablewidth}{0.9pt}}
\end{center}
\end{table}

 Table \ref{table3} lists the numerical results of the compact scheme \eqref{Csc1}--\eqref{Csc4} in time direction with fixed $h=\frac1{100}$, and Table \ref{table4} shows the numerical results in space direction with fixed $\tau=\frac1{5000}$ and different $\alpha$. The second-order accuracy in time and fourth-order accuracy in space are apparent in these two tables, respectively. We note that,
  in our implementation to the the non-global Lipschitz continuous term  $f(u)=u^3$, the condition $\tau=\nu h^{\frac12+\epsilon}$
  theoretically imposed in Remark \ref{condition} is not necessary.

\subsection{Comparison with other numerical schemes}
In this subsection, we give some comparisons between our proposed method and some standard methods for solving the equation \eqref{eq1}--\eqref{eq3}. It will be shown that our linearized schemes have advantages in both theoretical analysis and numerical computation.

One may construct the numerical schemes where time fractional derivative was approximate by the widely used classical $L_1$ formula \cite{SunANM2006} to solve this kind of equations. The scheme will take the form  as
\begin{align}\label{sch-compare}
\frac1{\mu}\left[a_0\delta_t u_i^{n+\frac12}-\sum_{k=1}^{n-1}(a_{n-k-1}-a_{n-k})\delta_t u_i^{k+\frac12}-a_{n-1}\psi_i \right]=\delta_x^2 u_i^{n+\frac12}-f(u_i^{n+\frac12})+p_i^{n+\frac12},
\end{align}
where $a_k=(k+1)^{2-\alpha}-k^{2-\alpha}$.

There are two commonly used approaches to approximate the nonlinear term $f(u_i^{n+\frac12})$. The first is the central approximation, i.e. $f(u_i^{n+\frac12})=\big[f(u_i^{n+1})+f(u_i^n)\big]/2+{\cal O}(\tau^2)$,
and the resulting scheme will be accurate of temporal order ${\cal O}(\tau^{3-\alpha})$, e.g. \cite{Vong2014,Vong2015,ChenH-Taiwan}, but this leads to nonlinear treatment so that iterative methods are required.
 However, the iterative methods would cause additional computational costs. Here we give a comparison by a numerical example, we apply our proposed scheme \eqref{sc1}--\eqref{sc4} and the scheme \eqref{sch-compare} with the above nonlinear approximation to solve the example in subsection \ref{Accuracy}, in which the nonlinear scheme is dealt with a fixed-point method.
 For simplicity of presentation, we only list results of {\bf Case 2} in Table \ref{table5}, remarking that similar results can be obtained for other cases and different parameters. From this table, one can clearly see that our scheme works more efficiently and spend less CPU (seconds) time than the iterative method. Moreover, when applying iterative methods for solving nonlinear schemes,
 the convergence of the iterative methods is usually not easy to be established.
  The second way to approximate nonlinear term is a linear approach, e.g. $f(u_i^{n+\frac12})=f(u_i^n)+{\cal O}(\tau)$.
 The stability and convergence can be verified like \cite{Dehghan2015}, but its convergence order is only ${\cal O}(\tau)$ or no more than ${\cal O}(\tau^{3-\alpha})$ in time.
   \begin{table}[hbt!]
 \begin{center}
 \caption{Numerical results of {\bf Case 2} by applying our scheme \eqref{sc1}--\eqref{sc4} and nonlinear scheme \eqref{sch-compare} with fixed-point iteration method, respectively, for $\alpha=1.8$ and $h=\frac1{1000}$.}\label{table5}
 \renewcommand{\arraystretch}{1}
 \def\temptablewidth{0.9\textwidth}
 {\rule{\temptablewidth}{0.9pt}}
 \begin{tabular*}{\temptablewidth}{@{\extracolsep{\fill}}ccccccc}
 $\tau$ &\multicolumn{3}{c}{scheme \eqref{sc1}--\eqref{sc4}}&\multicolumn{3}{c}{scheme \eqref{sch-compare} with iteration}\\
         \cline{2-4}                     \cline{5-7}
         &$E_2(\tau,h)$& Rate1   &CPU(s)   &$E_2(\tau,h)$& Rate1   &CPU(s)\\\hline
 $1/20$  & 4.7836e-03  & $\ast$  &0.35   & 1.3562e-02    & $\ast$  & 1.06  \\
 $1/40$  & 1.1950e-03  & 2.0011  &0.95   & 6.3733e-03    & 1.0895  & 3.14  \\
 $1/80$  & 2.9756e-04  & 2.0058  &2.24   & 2.8693e-03    & 1.1513  & 6.49  \\
 $1/160$ & 7.3449e-05  & 2.0184  &5.31   & 1.2624e-03    & 1.1845  & 13.38  \\
 $1/320$ & 1.7524e-05  & 2.0674  &11.29  & 5.4810e-04    & 1.2037  & 27.55
\end{tabular*}
{\rule{\temptablewidth}{0.9pt}}
\end{center}
\end{table}

One may also use the formula in Lemma \ref{lemma_time} to replace the $L_1$ formula and then achieve a ${\cal O}(\tau^2)$ in fractional approximation,
 %but the diffusion term will be $\delta_x^2 u^{n+\theta}_i$, which such that we can not take inner with $\delta_tu^{n+\frac12}$,
 but this gives rise to some difficulties in theoretical analysis and, to our knowledge,
 the linear method like that in \cite{Dehghan2015} may not be applied. Consequently,
 %is not suitable for dealing with $f(u_i^{n+\theta})$ and
 iterative methods are required.

All the discussions above illustrate that our proposed linearized schemes would be favorable to some usual numerical schemes.

\section{Concluding remarks}\label{Concluding}
 In this paper, we consider a nonlinear fractional order Klein-Gordon type equation.
 We proposed a linearized finite difference scheme to solve the problem numerically.
 The main advantage is that the nonlinear term is evaluated on previous time level.
 As a result, iterative method is not needed for implementation.
 However, shifting the evaluation to previous level causes difficulties in theoretical analysis.
 Inspired by some recent studies, we show that the scheme converges with second-order in time.
 The results are justified by some numerical tests.

\section*{Acknowledgment}
The authors would like to thank the referees for their comments which improve the paper significantly.
We also want to  thank Prof. Honglin Liao for helpful discussion on the estimates of Lemmas \ref{akbk} and \ref{dk}.

\end{document}